%% file: main.tex
\documentclass[twoside]{article}

\usepackage{bm}
\usepackage{hyperref}
\usepackage{url}
\usepackage{xspace}
\usepackage[algo2e,ruled]{algorithm2e}
\usepackage[authoryear,round]{natbib}
\usepackage{caption}
\usepackage{subcaption}
\usepackage{booktabs}
\usepackage{xcolor}
\usepackage{soul}
\input{cmd.tex}
\usepackage[accepted]{aistats2024}

\usepackage{soul}
\DeclareRobustCommand{\mhl}[1]{%
  \text{{$#1$}}%
}
\begin{document}

%

%











\twocolumn[

\aistatstitle{Extragradient Type Methods for Riemannian Variational Inequality Problems}

\aistatsauthor{ Zihao Hu \And Guanghui Wang \And  Xi Wang \And  Andre Wibisono \And Jacob Abernethy \And  Molei Tao }

\aistatsaddress{Georgia Tech \And  Georgia Tech \And AMSS, China \And Yale University \And  Georgia Tech \\Google Research  \And Georgia Tech }


]







\begin{abstract}
In this work, we consider monotone Riemannian Variational Inequality Problems (RVIPs), which encompass both Riemannian convex optimization and minimax optimization as particular cases. In Euclidean space,  the last-iterates of both the extragradient (EG) and past extragradient (PEG) methods converge to the solution of monotone variational inequality problems at a rate of $O\left(\frac{1}{\sqrt{T}}\right)$ \citep{cai2022finite}. However, analogous behavior on Riemannian manifolds remains  open. To bridge this gap, we introduce the Riemannian extragradient (REG) and Riemannian past extragradient (RPEG) methods. We show that both exhibit $O\left(\frac{1}{\sqrt{T}}\right)$ last-iterate convergence and $O\left(\frac{1}{{T}}\right)$ average-iterate convergence, aligning with observations in the Euclidean case. These results are enabled by judiciously addressing the holonomy effect so that additional complications in Riemannian cases can be reduced and the Euclidean proof inspired by the performance estimation problem (PEP) technique or the sum-of-squares (SOS) technique can be applied again. 
\end{abstract}

\section{Introduction}

Variational inequality problems (VIPs) \citep{kinderlehrer2000introduction,facchinei2003finite} play a pivotal role in mathematical programming, encompassing areas such as convex optimization and minimax optimization. Specifically, for the Euclidean space under the unconstrained setting, the objective of a VIP is to find  $\z^*$ satisfying:
\[
\textstyle\la F(\z^*),\z-\z^*\ra\geq 0\qquad\forall\z\in\R^d,
\]
where $F:\R^d\rightarrow\R^d$  represents an operator. In particular, by setting $\z=\z^*-\eta F(\z^*)$, the solution to an unconstrained VIP can be reduced to identifying the zeros of $F(\cdot)$. At first glance, it might seem intuitive to employ gradient descent (GD) given by
 $\z_{t+1}=\z_t-\eta F(\z_t)$ to solve a VIP.  However, while GD shows convergence for convex optimization tasks, it can, unfortunately, diverge for even monotone VIPs\footnote{Monotone VIP means the operator $F$ is monotone: $\la F(\z)-F(\z'),\z-\z'\ra\geq 0$ holds for any $\z,\z'\in\R^d$.}, irrespective of the step-sizes chosen \citep{facchinei2003finite}. Fortunately, sophisticated methods such as extragradient (EG) \citep{korpelevich1976extragradient}:
 \begin{equation}
 \begin{split}
   &\tz_t=\z_t-\eta F(\z_t)  \\
   &\z_{t+1}=\z_t-\eta F(\tz_t),
 \end{split}
 \end{equation}
and past extragradient (PEG) \citep{popov1980modification}:
\begin{equation}\label{eq:epeg}
\begin{split}
    &\tz_t=\z_t-\eta F(\tz_{t-1}) \\
    &\z_{t+1}=\z_t-\eta F(\tz_t).
\end{split}
\end{equation}
offer solutions to this hurdle. In the unconstrained domain, PEG is equivalent to the optimistic gradient descent ascent (OGDA) technique:
 \[
\tz_{t+1}=\tz_t-2\eta F(\tz_t)+\eta F(\tz_{t-1}).
\]
\begin{table*}[htbp]
    \centering
    \caption{Comparison of our results and previous extragradient type methods on the Riemannian gsc-convex gsc-concave saddle point problems,  where $\zeta$ and $\sigma$ are geometric constants arising from Riemannian cosine laws \citep{zhang2016first,alimisis2020continuous}. Notably, we achieve the first non-asymptotic last-iterate convergence for Riemannian extragradient type methods.}
    \label{tb:table1}
\begin{tabular}{cc}
\toprule[1pt]
Algorithm& Results   \\
\hline 
\addlinespace[3pt]
 RCEG \citep{zhang2022minimax} & average-iterate: $O\lt(\sqrt{\frac{\zeta}{\sigma}}\cdot\frac{1}{T}\rt)$, best-iterate: $O\lt(\frac{\sqrt{\zeta}}{\sigma}\cdot\frac{1}{\sqrt{T}}\rt)$ \\
\addlinespace[3pt]
 ROGDA \citep{wang2023riemannian} & average-iterate: $O\lt(\frac{\zeta}{\sigma}\cdot\frac{1}{T}\rt)$, best-iterate: $O\lt(\frac{\zeta}{\sqrt{\sigma^3}}\cdot\frac{1}{\sqrt{T}}\rt)$ \\
\addlinespace[3pt]
 REG (Theorem \ref{thm:reg}) & average-iterate: $O\lt(\frac{\zeta}{\sigma}\cdot\frac{1}{T}\rt)$, last-iterate: $O\lt(\frac{\zeta}{\sqrt{\sigma^3}}\cdot\frac{1}{\sqrt{T}}\rt)$ \\
\addlinespace[3pt]
 RPEG (Theorem \ref{thm:rpeg}) & average-iterate: $O\lt(\frac{\zeta}{\sigma}\cdot\frac{1}{T}\rt)$, last-iterate: $O\lt(\frac{\zeta}{\sqrt{\sigma^3}}\cdot\frac{1}{\sqrt{T}}\rt)$ \\
\bottomrule[1pt]
\end{tabular}
\end{table*}

Since extragradient type methods are easy to implement, relatively scalable to dimension, and have demonstrated pleasant empirical performance, they have become the standard tools for addressing VIPs and saddle point problems over the past few decades \citep{tseng2000modified,gidel2018variational,hsieh2019convergence}. At present, we understand that for convex-concave saddle point problems, both EG and PEG achieve \(O\left(\frac{1}{T}\right)\) \emph{average-iterate} convergence \citep{nemirovski2004prox,mokhtari2020convergence} and \(O\left(\frac{1}{\sqrt{T}}\right)\) \emph{last-iterate} convergence \citep{golowich2020last,golowich2020tight,gorbunov2022extragradient,gorbunov2022last,cai2022finite}.\footnote{The average-iterate and the last-iterate consider the convergence of $\bz_T\coloneqq\frac{1}{T}\sum_{t=1}^T\tz_t$ and $\z_T$, respectively.} The last-iterate convergence, although slower than the average-iterate convergence, offers two distinct advantages: (i) the last-iterate convergence is an appropriate performance metric even for non-convex-concave games, a condition not met by the average-iterate convergence due to the absence of Jensen's inequality; (ii) in practical scenarios like GAN training, the last-iterate exhibits strong empirical results \citep{daskalakis2018training,chavdarova2019reducing}.

Meanwhile, in recent years, Riemannian convex optimization and minimax optimization have attracted considerable interest \citep{zhang2016first,alimisis2020continuous, ahn2020nesterov,kim2022nesterov,zhang2022minimax,jordan2022first,martinez2023minimax}. However, Riemannian variational inequality problems (RVIPs), the generalized counterpart of Riemannian convex optimization and minimax optimization,
remain relatively underexplored. In this paper, we consider extragradient type methods for the following RVIP: identify
$\z^*$ such that
\[
\la F(\z^*),\invexp_{\z^*}\z\ra\geq 0,\qquad\forall\z\in\M.
\]
Here, $\M$ denotes a $d$-dimensional Riemannian manifold and $F(\cdot)$ is a vector field defined on $\M$. For Riemannian convex-concave saddle point problems, \citet{zhang2022minimax} introduce the Riemannian corrected extragradient (RCEG), while \citet{wang2023riemannian} propose the Riemannian optimistic gradient descent ascent (ROGDA), both achieving $O\left(\frac{1}{T}\right)$ average-iterate convergence. However, the non-asymptotic \emph{ last-iterate} convergence of these methods remains unexplored except for strongly geodesically convex-concave problems \citep{jordan2022first,wang2023riemannian}. This naturally raises the question:

\emph{Do there exist Riemannian analogs of EG and PEG that concurrently exhibit non-asymptotic average-iterate and last-iterate convergence behaviors?}

In this study, we confirm this question and outline the following contributions.
\begin{itemize}
    \item We introduce \underline{R}iemannian \underline{e}xtra\underline{g}radient (REG) and \underline{R}iemannian \underline{p}ast \underline{e}xtra\underline{g}radient (RPEG) as novel first-order methods tailored for monotone RVIPs.
    \item Both REG and RPEG, as detailed in Theorems \ref{thm:eglast} and \ref{thm:peglast},  exhibit $O\left(\frac{1}{\sqrt{T}}\right)$ last-iterate convergence in the context of monotone RVIPs.
    \item In the realm of Riemannian minimax optimization, both REG and RPEG achieve $O\left(\frac{1}{\sqrt{T}}\right)$ last iterate convergence and $O\left(\frac{1}{T}\right)$ average-iterate convergence, as delineated in Table \ref{tb:table1}.
\end{itemize}
For completeness, we have included the proof of the best-iterate convergence for RCEG in Appendix \ref{app:best}. It is beneficial to compare the convergence rates of the algorithms listed in Table \ref{tb:table1}, focusing on the geometric constants $\zeta$ and $\sigma$. Given that $\zeta\geq\sigma$, the average-iterate convergence rate of RCEG algorithm comes with a more favorable geometric constant compared to our REG algorithm. Meanwhile, the geometric constants of ROGDA show similarities to those of our REG and RPEG algorithms. This similarity may arise because all those methods rely on bounding the holonomy distortion. An intriguing open question is how to enhance the last-iterate convergence of  REG and RPEG by optimizing the curvature constants.





We now outline the primary technical challenges and our solutions. State-of-the-art proofs validating the last-iterate convergence of EG and PEG are generally inspired by either the performance estimation problem (PEP) approach \citep{gorbunov2022extragradient,gorbunov2022last} or the sum-of-squares (SOS) technique \citep{cai2022finite}. At their core, both methods cast the estimation of an optimization algorithm's convergence rate as an optimization problem, subsequently obtaining a numerical solution through sequential convex relaxation. This solution inherently offers insights into crafting a proof. However, applying the PEP and SOS techniques directly to the manifold setting proves difficult, largely because they intrinsically depend on the ``interpolation condition" \citep{taylor2017smooth}, which is inherently tied to the Euclidean space, not geodesic metric spaces\footnote{For a deeper exploration of the ``geodesically convex interpolation", readers are referred to \citet[Section 8]{criscitiello2023curvature}.}.

Yet, this challenge does not inherently restrict us from leveraging the insights gleaned from the PEP or SOS methods. In our study, we meticulously craft REG and RPEG based on these insights, demonstrating that validating the last-iterate convergence in manifold contexts can largely mirror the proofs in the Euclidean setting, provided that the \emph{holonomy effect}\footnote{To put it simply, after a vector undergoes parallel transport along a geodesic loop, the end result differs from its original form, a phenomenon termed holonomy.} is handled with care. We posit that this novel approach not only resolves our immediate challenges but could potentially benefit a broader range of Riemannian optimization problems in the future.

\section{Related Work}
In this section, we briefly review prior research on extragradient type algorithms in the Euclidean space and Riemannian minimax optimization.

\textbf{Extragradient Type Methods in Euclidean Space.} \citet{nemirovski2004prox} demonstrate an $O\left(\frac{1}{T}\right)$ average-iterate convergence rate of EG with respect to the primal-dual gap. A comparable result for OGDA is documented by \citet{mokhtari2020convergence}. Building on additional assumptions, specifically the Lipschitz continuity of the Jacobian of $ F $, \citet{golowich2020last, golowich2020tight} are the first to establish an $O\left(\frac{1}{\sqrt{T}}\right)$ last-iterate convergence rate for EG and PEG. This milestone is further developed by \citet{gorbunov2022extragradient}, who bypass the aforementioned assumption and showcase an $O\left(\frac{1}{\sqrt{T}}\right)$ last-iterate convergence using the PEP technique. Various works such as \citet{cai2022finite} and \citet{gorbunov2022last} extend these results to constrained settings.  One intriguing discrepancy that emerges is between the average-iterate convergence $O\left(\frac{1}{T}\right)$ and the last-iterate convergence $O\left(\frac{1}{\sqrt{T}}\right)$ rates. It raises the question: are there accelerated first-order methods that can achieve $O\left(\frac{1}{T}\right)$ last-iterate convergence? Drawing inspiration from the Halpern iteration \citep{lieder2021convergence}, \citet{yoon2021accelerated} introduce the extra anchored gradient (EAG), which achieves $O\left(\frac{1}{T}\right)$ last-iterate convergence rate. 

\textbf{Riemannian Minimax Optimization.} \citet{zhang2022minimax} propose RCEG, which achieves $O\left(\frac{1}{T}\right)$ average-iterate convergence for geodesically convex-concave problems. Our contributions, REG and RPEG, surpass this by attaining both $O\left(\frac{1}{\sqrt{T}}\right)$ last-iterate and $O\left(\frac{1}{T}\right)$ average-iterate convergence rates. \citet{hu2023minimizing} consider how to make RCEG work in the online improper learning setting. \citet{han2023riemannian} put forth the Riemannian Hamiltonian gradient descent, which exhibits linear convergence under the Riemannian Polyak-Lojasiewicz condition. However, this condition predominantly applies to  strongly geodesically convex-concave problems. Other notable contributions include \citet{jordan2022first}, who demonstrate linear last-iterate convergence of RCEG for strongly convex-concave settings, and \citet{wang2023riemannian}, who propose ROGDA. This method achieves $O\left(\frac{1}{T}\right)$ average-iterate convergence and linear last-iterate convergence in convex-concave and strongly convex-concave settings, respectively. It should be noted that while \citet{wang2023riemannian} rely on quantifying the holonomy effect in geodesic quadrilaterals, our approach capitalizes on the analogous effect in geodesic triangles. As we will see later, this is a key observation that enables us to establish the last-iterate convergence within the Riemannian context.  The Riemannian gradient descent ascent (RGDA) in the deterministic setting and the stochastic setting has been considered by \citet{jordan2022first,huang2023gradient}. By leveraging the idea of converting minimax optimization to sequential strongly convex optimization problems, \citet{martinez2023minimax} introduce doubly-looped algorithms specifically designed for Hadamard manifolds. Furthermore, the algorithms proposed in \citet{martinez2023minimax} demonstrate accelerated convergence rates and inherently address the constrained scenario. The quest for singly-looped Riemannian first-order minimax optimization algorithms featuring accelerated rates continues to be a compelling area of investigation. \citet{cai2023curvature} show a curvature-independent linear last-iterate convergence of Riemannian gradient descent for strongly monotone RVIPs. Our approach offers a slower $O\lt(\frac{1}{\sqrt{T}}\rt)$  last-iterate convergence rate, but applies to general monotone RVIPs. 

After finalizing the camera-ready version of our work, we would like to acknowledge the Riemannian stochastic approximation framework proposed by \citet{karimi2022dynamics}. Their framework covers the stochastic versions of our REG and RPEG as special cases, referred to as RSEG and ROG. However, it is important to note that while \citet{karimi2022dynamics} focus on the stochastic setting and aim to establish asymptotic convergence, our work focuses on the deterministic setting and establishes non-asymptotic convergence.
\section{Preliminaries}
In this section, we provide an overview of Riemannian geometry, RVIPs, and Riemannian minimax optimization. We also introduce some assumptions that are necessary to establish our key results.

\subsection{Riemannian Geometry } We outline the foundational concepts of Riemannian geometry here. For a more in-depth exposition, the reader is directed to \citet{petersen2006riemannian,lee2018introduction}. We consider a $d$-dimensional smooth manifold $\M$, a topological space where every point possesses an open neighborhood that can be smoothly mapped to an open subset in $\R^d$. For each point $\x$ on the manifold $\M$, there are $d$ directions (tangent vectors). Beginning at $\x$ and moving infinitesimally in any of these directions remains within $\M$. The tangent space at $\x$, symbolized as $T_{\x}\M$, is a vector space comprising all these tangent vectors. A Riemannian manifold is a smooth manifold equipped with a continuously differentiable Riemannian metric. For any point $\x\in\M$, this metric allows the calculation of the inner product $\la\u,\v\ra_{\x}$ and the magnitude $\|\u\|_{\x}=\sqrt{\la\u,\u\ra_{\x}}$ for tangent vectors $\u,\v\in T_{\x}\M$. We sometimes omit the reference point $\x$ when it is clear from the context.

A geodesic segment connecting two points $\x,\y\in\M$ is a constant-speed curve that locally minimizes the distance between $\x$ and $\y$, serving as a natural extension of line segments in Euclidean space. Formally, $\gm(t):[0,1]\rightarrow\M$ represents a geodesic segment, with $\gm(0)=\x$, $\gm(1)=\y$, and its initial velocity given by $\dot{\gm}(0)=\v\in T_{\x}\M$. The exponential map, $\expmap_{\x}(\cdot)$, transitions from a tangent space to the manifold, while the inverse exponential map, $\invexp_{\x}(\cdot)$, maps from the manifold to a tangent vector. For the aforementioned geodesic segment $\gm(t)$, $\expmap_{\x}\v=\y$ and $\invexp_{\x}\y=\v$ hold true. The Riemannian distance, $d(\x,\y)$, quantifies the geodesic distance between $\x$ and $\y$. Given the constant speed of the geodesic, we have $d(\x,\y)=\|\invexp_{\x}\y\|$.
For two distinct points $\x$ and $\y$ and a tangent vector $\u\in T_{\x}\M$, the parallel transport operation, denoted by $\G_{\x}^{\y}\u$, smoothly shifts $\u$ from being an element in $T_{\x}\M$ to being an element in $T_{\y}\M$ via the geodesic joining $\x$ and $\y$. This operation preserves both the inner product and the norm of the tangent vector.


The curvature of a Riemannian manifold is precisely defined by the Riemannian curvature tensor. For practicality, the sectional curvature is used more frequently in machine learning \citep{zhang2016first,ahn2020nesterov,kim2022nesterov}. The sectional curvature at any point $\x\in\M$ relies on all $2$-planes in $T_{\x}\M$. On Riemannian manifolds with positive, zero, and negative sectional curvatures, geodesics that start out parallel will, respectively, converge, remain parallel, and diverge. A class of Riemannian manifolds of particular interest is the Hadamard manifolds, which are complete and simply connected spaces with non-positive curvature. There, any pair of distinct points are connected by a globally length-minimizing geodesic.

A \textit{geodesically-convex} (gsc-convex) set contains all length-minimizing geodesics connecting two distinct points within the set. Let $\mathcal{N} \subseteq \mathcal{M}$ be a gsc-convex set. A function $f: \mathcal{N} \to \mathbb{R}$ is termed gsc-convex if and only if $f$ is convex when restricted to any geodesic with the minimum length connecting two distinct points in $\mathcal{N}$. Formally, for all geodesic paths $\gamma(t) \subseteq \mathcal{N}$,
\[
\textstyle f(\gamma(t)) \leq (1-t) f(\gamma(0)) + t f(\gamma(1)).
\]
For differentiable functions, geodesic convexity can be represented as:
\[
\textstyle f(\mathbf{y}) \geq f(\mathbf{x}) + \langle \nabla f(\mathbf{x}), \invexp_{\mathbf{x}} \mathbf{y} \rangle, \quad \forall \mathbf{x}, \mathbf{y} \in \mathcal{N},
\]
where $\nabla f(\mathbf{x}) \in T_{\mathbf{x}}\mathcal{M}$ is the Riemannian gradient. Consequently, the notion of geodesically-smooth (gsc-smooth) functions emerges. A function $f$ is termed $L$-gsc-smooth if 
\[
\textstyle \|\nabla f(\mathbf{x}) - \Gamma_{\mathbf{y}}^{\mathbf{x}} \nabla f(\mathbf{y})\| \leq L\cdot d(\mathbf{x}, \mathbf{y}),
\]
or equivalently, for all $\mathbf{x}, \mathbf{y} \in \mathcal{N}$:
\[
\textstyle f(\mathbf{y}) \leq f(\mathbf{x}) + \langle \nabla f(\mathbf{x}), \invexp_{\mathbf{x}} \mathbf{y} \rangle + \frac{L}{2} d(\mathbf{x}, \mathbf{y})^2.
\]
The \emph{holonomy} effect, originating from the curvature of the Riemannian manifold, captures the ``turning" of a vector as it undergoes parallel transport around a geodesic loop. As depicted in Figure \ref{fig:holo}, beginning with a tangent vector $\mathbf{u} \in T_{\mathbf{x}}\mathcal{M}$, and translating $\mathbf{u}$ along geodesic segments $\overline{\mathbf{x}\mathbf{y}}$, $\overline{\mathbf{y}\mathbf{z}}$ and $\overline{\mathbf{z}\mathbf{x}}$, upon returning to $\mathbf{x}$, the vector $\Gamma_{\mathbf{z}}^{\mathbf{x}}\Gamma_{\mathbf{y}}^{\mathbf{z}}\Gamma_{\mathbf{x}}^{\mathbf{y}}\mathbf{u}$, though still in $T_{\mathbf{x}}\mathcal{M}$, deviates in direction from the initial tangent vector $\mathbf{u}$. Although quantifying the holonomy effect for general geodesic loops is complex, this work demonstrates that approximating the holonomy effect on a geodesic triangle is sufficient to ascertain the last-iterate convergence of Riemannian extragradient type methods.

\begin{figure}[htbp]
    \centering
    \includegraphics[scale=.8]{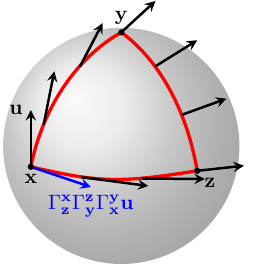}
    \caption{An illustration of the holonomy effect on a sphere.}
    \label{fig:holo}
\end{figure}
\subsection{RVIPs and Riemannian Minimax Optimization} 
Recall the definition of the Riemannian VIP is to seek a point $\mathbf{z}^*$ such that
\begin{equation}
\langle F(\mathbf{z}^*), \invexp_{\mathbf{z}^*} \mathbf{z} \rangle \geq 0, \quad \forall \mathbf{z} \in \mathcal{M},
\end{equation}
which is equivalent to finding $\mathbf{z}^*$ for which $F(\mathbf{z}^*) = 0$ by substituting $\mathbf{z} = \expmap_{\mathbf{z}^*}(-\eta F(\mathbf{z}^*))$. We demonstrate that Riemannian convex optimization and minimax optimization are special cases of Riemannian VIPs in the sequel.

A Riemannian convex optimization problem is given by
\begin{equation}
\min_{\mathbf{z} \in \mathcal{M}} f(\mathbf{z}),
\end{equation}
where $f(\mathbf{z})$ is a gsc-convex function on a Riemannian manifold $\mathcal{M}$. The corresponding RVIP is obtained by selecting $F(\mathbf{z}) = \nabla f(\mathbf{z})$. Riemannian convex optimization finds applications in operator scaling \citep{allen2018operator}, Gaussian mixture models \citep{hosseini2015matrix}, and the calculation of the Fréchet mean \citep{lou2020differentiating}.  Another additional remark is, most algorithms for convex optimization still work in general, nonconvex settings, and it is just that their quantitative convergence guarantees may not carry through. General Riemannian optimizations lead to even more applications, such as large scale eigenvalue/PCA/SVD problems \cite{tao2020variational}, generic improvement of transformer and approximation of optimal transport (Wasserstein) distance in high dimensions \cite{kong2022momentum}.

A Riemannian minimax problem can be articulated as
\begin{equation}
\min_{\mathbf{x} \in \mathcal{M}_1} \max_{\mathbf{y} \in \mathcal{M}_2} f(\mathbf{x},\mathbf{y}),
\end{equation}
where $\M_1$ and $\M_2$ are Riemannian manifolds and $f$ is gsc-convex in $\mathbf{x}$ and gsc-concave in $\mathbf{y}$. Examples in this category encompass Riemannian constrained convex optimization, robust geometry-aware PCA, and robust matrix Fréchet mean computation \citep{zhang2022minimax, jordan2022first}. Adopting $\z= \big(\begin{smallmatrix}
  \x\\
  \y
\end{smallmatrix}\big)$ and $F(\z)= \lt(\begin{smallmatrix}
  \nabla_{\x}f(\x,\y)\\
  -\nabla_{\y}f(\x,\y)
\end{smallmatrix}\rt)$,
it is evident that $\mathbf{z}^* = \lt(\begin{smallmatrix} \mathbf{x}^* \\ \mathbf{y}^* \end{smallmatrix}\rt)$, the RVIP solution, stands as a saddle point of $f(\mathbf{x}, \mathbf{y})$.

For (unconstrained) RVIP which seeks a $\mathbf{z}^*$ such that $F(\mathbf{z}^*) = 0$, the norm $\|F(\mathbf{z}_t)\|$ serves as the convergence criterion, aligning with standard proofs of last-iterate convergence in the Euclidean domain \citep{golowich2020last, gorbunov2022extragradient, cai2022finite}. However, for constrained situations, the norm $\|F(\mathbf{z}_t)\|$ is unsuitable since the RVIP solution does not inherently ensure $F(\mathbf{z}^*) = 0$. In Euclidean space, an alternative concept known as the ``tangent residual'' has been shown to be effective in constrained cases, as detailed in \citep{cai2022finite}. It would be intriguing to explore the applicability of this technique within the Riemannian context.

For Riemannian saddle point problems, a convergence measure analogous to $\|F(\mathbf{z}_t)\|$ is the Riemannian Hamiltonian:

\begin{defn}[Riemannian Hamiltonian]
For a geodesically convex-concave objective $f$ defined on $\mathcal{M}_1 \times \mathcal{M}_2$, the Riemannian Hamiltonian at $(\mathbf{x},\mathbf{y})$ is
\begin{equation*}
\text{Ham}_f(\mathbf{x},\mathbf{y}) \coloneqq \|\nabla_{\mathbf{x}} f(\mathbf{x},\mathbf{y})\|^2 + \|\nabla_{\mathbf{y}} f(\mathbf{x},\mathbf{y})\|^2 = \|F(\mathbf{z})\|^2.
\end{equation*}
\end{defn}
An alternative convergence metric introduces the Riemannian primal-dual gap:

\begin{defn}[Riemannian Primal-dual Gap]
Assume $f: \mathcal{M}_1 \times \mathcal{M}_2 \to \mathbb{R}$ is geodesically convex-concave, and sets $\mathcal{X} \subseteq \mathcal{M}_1$ and $\mathcal{Y} \subseteq \mathcal{M}_2$ are gsc-convex and compact. The primal-dual gap on $\mathcal{X} \times \mathcal{Y}$ is defined as:
\begin{equation*}
\textstyle \text{Gap}_f^{\mathcal{X} \times \mathcal{Y}}(\mathbf{x},\mathbf{y}) = \max_{\mathbf{y}' \in \mathcal{Y}} f(\mathbf{x},\mathbf{y}') - \min_{\mathbf{x}' \in \mathcal{X}} f(\mathbf{x}',\mathbf{y}).
\end{equation*}
\end{defn}
This gap measures how the utility derived by one participant changes upon unilateral action modification. It is imperative to define the primal-dual gap on compact sets $\mathcal{X} \times \mathcal{Y}$ to ensure the quantity is bounded.
\subsection{Assumptions}
In this part, we provide some key assumptions that will be essential to our later results.
\begin{ass}\label{ass:rm}
   Let $\M$ be a $d$-dimensional complete and  simply connected Riemannian manifold with sectional curvature
   lower bounded by $\kappa$ and upper bounded by $K$. Assume $F(\cdot)$ is a vector field on $\M$ and $F(\z)=0$ admits a solution $\z^*$. We denote $D$ as an upper bound of $d(\z_0,\z^*)$. When $K>0$, we require that $D\leq\frac{4\pi}{9\sqrt{K}}$.   
\end{ass}
\begin{defn}\label{def:def}
Under Assumption \ref{ass:rm}, we define $K_m=\max\{|\k|,|K|\}$ and $\D=\{\z|d(\z,\z^*)\leq\frac{6D}{5}\}$. We also define $\bs=\sigma\lt(K,\frac{91D}{81}\rt)$ and $\bzt=\zeta\lt(\kappa,\frac{7D}{5}\rt)$, where $\sigma(K,\cdot)$ and $\zeta(\k,\cdot)$ are  geometric constants defined in Lemmas \ref{lem:cos2} and \ref{lem:hessian} (Appendix \ref{app:tech}), respectively.
\end{defn}
\begin{remark}
{Note that $D<\frac{\pi}{2\sqrt{K}}$ is required to guarantee the unique geodesic property} \citep[Theorem 5.14]{cheeger1975comparison} and our condition \(D\leq \frac{4\pi}{9\sqrt{K}}\) is substantially close to \(D<\frac{\pi}{2\sqrt{K}}\).
\end{remark}
\begin{ass}\label{ass:mono2}
    $F(\cdot)$ is monotone on  $\D$:
    $$
    \la \G_{\z'}^{\z} F(\z')-F(\z),\invexp_{\z}\z'\ra\geq 0.
    $$
\end{ass}
\begin{ass}\label{ass:lips2}
    $F(\cdot)$ is $L$-Lipschitz on  the manifold $\M$, which means
    $$
    \|F(\z)-\G_{\z'}^{\z} F(\z')\|\leq L\cdot d(\z,\z').
    $$
\end{ass}
\begin{ass}\label{ass:grad}
   On the set $\D$, the norm of $F(\cdot)$ is bounded: 
    \[
    \|F(\z)\|\leq G.
    \]
\end{ass}
\begin{remark}
We utilize the fact that $F(\z^*)=0$, which leads to
\[
\|F(\z)\|=\|F(\z)-\G_{\z^*}^{\z}F(\z^*)\|\leq L\cdot d(\z,\z^*).
\]
If the algorithm's trajectory remains bounded, it directly implies an upper bound on $\|F(\z)\|$. Initiating with $\z_0$ where $d(\z_0, \z^*) \leq D$, it can be verified (see Corollaries \ref{cor:eg} and \ref{cor:peg} for details) that all iterates of both REG and RPEG are confined within
\[
\textstyle\D=\lt\{\z \mid d(\z, \z^*) \leq \frac{6D}{5}\rt\},
\]
allowing us to dispense with Assumption \ref{ass:grad} without sacrificing generality.

\end{remark}


\section{Riemannian Extragradient}
\label{sec:reg}
In this part, we discuss the last-iterate and average-iterate convergence of the Riemannian extragradient method. The details of the proof for this section are deferred to Appendix \ref{app:reg}. As a precursor, we revisit the proof in the Euclidean space.
\subsection{Warm-up: the Euclidean Case}
The \emph{extragradient} (EG) method is pivotal for tackling saddle point and variational inequality problems. Given an operator $F:\R^d\rightarrow\R^d$, the EG algorithm updates as follows at each iteration:
\begin{equation}\label{eq:eg}
    \z_{t+1}=\z_{t}-\eta F(\tz_{t})\coloneqq \z_{t}-\eta F(\z_{t}-\eta F(\z_{t})).
\end{equation}
When $F(\cdot)$ is monotone and $L$-Lipschitz, as detailed in Assumptions \ref{ass:mono} and \ref{ass:lips}, EG enjoys $O\left(\frac{1}{T}\right)$ average-iterate convergence \citep{nemirovski2004prox,mokhtari2020convergence} and $O\left(\frac{1}{\sqrt{T}}\right)$ last-iterate convergence rate \citep{golowich2020last,gorbunov2022extragradient,cai2022finite}.

\begin{ass}\label{ass:mono}
    $F(\cdot)$ is monotone, which means
    $$
    \la F(\z)-F(\z'),\z-\z'\ra\geq 0\quad\forall\z,\z'\in\mathbb{R}^d.
    $$
\end{ass}
\begin{ass}\label{ass:lips}
    $F(\cdot)$ is $L$-Lipschitz, which means
    $$
    \|F(\z)-F(\z')\|\leq L\|\z-\z'\|\quad\forall\z,\z'\in\mathbb{R}^d.
    $$
\end{ass}
 The analysis of the last-iterate convergence of EG contains two parts:
\begin{itemize}
    \item (Best-iterate convergence) There exists $t'\in [T]$:
    $$
    \textstyle\|F(\z_{t'})\|\leq O\lt(\frac{1}{\sqrt{T}}\rt).
    $$
    \item (Non-increasing operator norm) For any $t\in[T]$,
    $$
    \textstyle\|F(\z_{t+1})\|\leq \|F(\z_{t})\|.
    $$
\end{itemize}
While the best-iterate convergence of the extragradient method is well-established \citep{korpelevich1976extragradient,nemirovski2004prox}, the proof on the monotonicity of the operator norm is far from obvious. By introducing an additional assumption that the Jacobian of $F$ is $\Lambda$-Lipschitz, \citet{golowich2020last} demonstrate a marginally weaker inequality:
\[
    \textstyle\|F(\z_{t+1})\|\leq (1+\epsilon_t)\|F(\z_{t})\|.
\]
where $\epsilon_t$ is small, thus establishing the $O\left(\frac{1}{\sqrt{T}}\right)$ last-iterate convergence. Subsequent studies by \citet{gorbunov2022extragradient,cai2022finite} eliminate this additional assumption by resorting to the Performance Estimation Problem (PEP) and the Sum-of-Squares (SOS) technique. The following lemma lays the groundwork for these findings.


\begin{lemma}\citep{cai2022finite}\label{lem:egpep}
    Suppose
    \begin{equation}\label{eq:peg}
    \begin{aligned}
&0  \leq\left\langle F\left(\z_{t}\right)-F\left(\z_{t+1}\right), F\left(\tz_{t}\right)\right\rangle, \\
&\left\|F\left(\tz_{t}\right)-F\left(\z_{t+1}\right)\right\|^2  \leq L^2 \eta^2\left\|F\left(\tz_{t}\right)-F\left(\z_{t}\right)\right\|^2
\end{aligned}
\end{equation}
hold, where $L^2 \eta^2\leq 1$, then $\|F(\z_{t+1})\|\leq \|F(\z_t)\|$.
\end{lemma}
In Euclidean case, Equation \eqref{eq:peg} can be derived from Equation \eqref{eq:eg}, Assumptions \ref{ass:mono} and \ref{ass:lips}.

\subsection{Riemannian Extragradient and Convergence Rates}
A natural inquiry arises: what is the last-iterate convergence of the Riemannian extragradient for monotone RVIPs? In this section, we demonstrate that the answer remains $O\lt(\frac{1}{\sqrt{T}}\rt)$.

Intuitively, we aim to identify an analog of Equation \eqref{eq:peg} in the Riemannian setting. To this end, we propose Riemannian extragradient (REG):
\begin{equation}\label{eq:reg}
        \begin{split}
        &\tz_{t}=\expmap_{\z_{t}}(-\eta F(\z_{t}))\\
        &\z_{t+1}=\expmap_{\z_{t}}(-\eta \Gamma_{\tz_{t}}^{\z_{t}}F(\tz_{t})),
    \end{split}
    \end{equation}
which employs the parallel transport to respect the Riemannian metric. In the high level, given REG's update rule, we can establish a Riemannian analog of Equation \eqref{eq:peg} and invoke Lemma \ref{lem:egpep} to show
 the operator norm $\|F(\z_t)\|$ is still non-increasing.

We first delve into the demonstration of the $O\lt(\frac{1}{\sqrt{T}}\rt)$  best-iterate convergence of REG. Since Assumptions~\ref{ass:mono2} and~\ref{ass:grad} are only valid on $\D$ rather than the entire manifold $\M$, it is crucial to ensure that all iterates of REG remain bounded. This fact is demonstrated with the aid of Lemma~\ref{lem:egbest} and Corollary \ref{cor:eg}.
Apart from establishing the boundedness of $\z_t$, Lemma~\ref{lem:egbest} also suggests an $O\left(\frac{1}{\sqrt{T}}\right)$ best-iterate convergence of REG.

\begin{lemma}\label{lem:egbest}
   Under Assumptions \ref{ass:rm},  \ref{ass:lips2} and \ref{ass:grad}. For the iterates of REG as in Equation \eqref{eq:reg},
   \[
   d(\z_t,\z^*)\leq D
   \]
   holds for any $t\geq 0$ and 
   \[
    \eta\leq\min\lt\{\frac{1}{\sqrt{8L^2+306K_mG^2}},\frac{\bs}{56\sqrt{K_m}DL+8\bzt L+\bs L} \rt\}.
    \]
    Also, we can show there exists $t'\in[T]$ such that
    \[
    \begin{split}
        \textstyle\|F(\z_{t'})\|=\mhl{O\lt(\frac{\bzt}{\sqrt{\bs^3T}}\rt)}.        
    \end{split}
    \]   
\end{lemma}
Therefore, under mild conditions, we can ensure that the trajectory of REG remains bounded.
\begin{cor}\label{cor:eg}
       Under Assumptions \ref{ass:rm},  \ref{ass:lips2} and \ref{ass:grad}. For the iterates of REG as in Equation \eqref{eq:reg}, $\z_t,\tz_t\in\D$
   holds for any $t\geq 0$ and 
   \[
    \textstyle\eta\leq\min\lt\{\frac{1}{\sqrt{8L^2+306K_mG^2}},\frac{\bs}{56\sqrt{K_m}DL+8\bzt L+\bs L} \rt\},
    \]
    where $\D$ is defined in Definition \ref{def:def}.
\end{cor}
To establish the last-iterate convergence of REG, we need to show the operator norm $\|F(\z_t)\|$ does not increase. The following lemma, which quantifies the holonomy effect on a geodesic triangle, plays a crucial role in achieving that objective.

\begin{lemma}\label{lem:holo2}
    (Informal) For a small geodesic triangle $\triangle \x\y\z$ on a Riemannian manifold $\M$ and $\u\in T_{\x}\M$, we have
    \[
    \|\G_{\z}^{\x}\G_{\y}^{\z}\G_{\x}^{\y}\u-\u\|=O(1)\cdot \|\u\|\cdot d(\y,\z).
    \]
\end{lemma}
\begin{remark}
It is beneficial to recognize a more potent implication of Lemma \ref{lem:holo2}: 
\[
    \|\G_{\z}^{\x}\G_{\y}^{\z}\G_{\x}^{\y}\u-\u\|=O(1)\cdot\min\{d(\x,\y),d(\x,\z),d(\y,\z)\}
    \]
holds for a small geodesic triangle, given that parallel transport preserves the norm of a tangent vector.
\end{remark}
A formal description and proof of Lemma \ref{lem:holo2} can be found in Appendix \ref{app:holo}. For REG, with a small step-size $\eta$, we demonstrate in the subsequent lemma that the distortion due to the holonomy effect is small and the operator norm does not increase.
\begin{lemma}\label{lem:decreasenorm}
   Under Assumptions \ref{ass:rm}, \ref{ass:mono2}, \ref{ass:lips2}, \ref{ass:grad}. For the iterates of REG as in Equation \eqref{eq:reg},
    we can show
    \[
    \begin{split}
        \|F(\z_{t+1})\|\leq\|F(\z_{t})\|,
    \end{split}
    \]
    by choosing
    \[
    \textstyle\eta\leq\min\lt\{\frac{1}{\sqrt{8L^2+306K_mG^2}},\frac{\bs}{56\sqrt{K_m}DL+8\bzt L+\bs L} \rt\}.
    \]
\end{lemma}

Now we are ready to present the last-iterate convergence of REG.
\begin{thm}\label{thm:eglast}
    Under Assumptions \ref{ass:rm}, \ref{ass:mono2}, \ref{ass:lips2}, \ref{ass:grad}. For the iterates of REG as in Equation \eqref{eq:reg},
    we can choose
    \[
  \textstyle  \eta=\min\lt\{\frac{1}{\sqrt{8L^2+306K_mG^2}},\frac{\bs}{56\sqrt{K_m}DL+8\bzt L+\bs L} \rt\}
    \]    
to achieve $O\lt(\frac{1}{\sqrt{T}}\rt)$ last-iterate convergence for the monotone variational problem with an operator $F$.    
\end{thm}

\begin{remark}
It is instructive to compare the step-size from Theorem \ref{thm:eglast} against that of \citet{gorbunov2022extragradient}. When we consider the Riemannian manifold as an Euclidean space, our required step-size is $\eta\leq\frac{1}{8L}$, whereas \citet{gorbunov2022extragradient} prescribes $\eta\leq\frac{1}{\sqrt{2}L}$. This disparity hints that the constants in our findings may not be the tightest.
\end{remark}
When we consider Riemannian saddle point problems, Theorem \ref{thm:reg} shows REG attains $O\lt(\frac{1}{\sqrt{T}}\rt)$ last-iterate convergence and $O\lt(\frac{1}{T}\rt)$ average-iterate convergence, simultaneously.
\begin{thm}\label{thm:reg}
    Consider a Riemannian minimax optimization problem 
    \[
    \min_{\x\in\M_1}\max_{\y\in\M_2}f(\x,\y)
    \]
    where $f$ is geodesically convex-concave. Let $\M\coloneqq\M_1\times\M_2$, $\z^*\coloneqq \big(\begin{smallmatrix}
  \x^*\\
  \y^*
\end{smallmatrix}\big)$ to be the saddle point, and $\X=\B\lt(\x^*,\frac{\sqrt{2}D}{2}\rt)$, $\Y=\B\lt(\y^*,\frac{\sqrt{2}D}{2}\rt)$ to be geodesic balls.  Under Assumptions \ref{ass:rm}, \ref{ass:mono2}, \ref{ass:lips2}, \ref{ass:grad}, if we apply REG in Equation \eqref{eq:reg} with
\[
    \textstyle\eta=\min\lt\{\frac{1}{\sqrt{8L^2+306K_mG^2}},\frac{\bs}{56\sqrt{K_m}DL+8\bzt L+\bs L} \rt\},
    \]
 then
\[
\textstyle\max_{\y\in\Y}f(\x_T,\y)-\min_{\x\in\X}f(\x,\y_T)=\mhl{O\lt(\frac{\bzt}{\sqrt{\bs^3T}}\rt)},
\]
and
\[
\textstyle\max_{\y\in\Y}f(\bx_T,\y)-\min_{\x\in\X}f(\x,\by_T)= \mhl{O\lt(\frac{\bzt}{\bs T}\rt)},
\]
where $\bx_T=\expmap_{\bx_{T-1}}\lt(\frac{1}{T}\invexp_{\bx_{T-1}}\tx_T\rt)$ and $\by_T=\expmap_{\by_{T-1}}\lt(\frac{1}{T}\invexp_{\by_{T-1}}\ty_T\rt)$ are the geodesic ergodic averages of $\tx_t$ and $\ty_t$ for $t=1,\dots, T$.
\end{thm}




Up to now, a natural question that might arise is whether the technique regarding holonomy distortion suffices to establish the last-iterate convergence for RCEG \citep{zhang2022minimax} or ROGDA \citep{wang2023riemannian}. We explored the possibility of establishing the last-iterate convergence of RCEG Appendix \ref{app:rceg}. The issue appears to involve a new distortion term and we use an alternative update $\z_{t+1}=\expmap_{\z_{t}}(-\eta \Gamma_{\tz_{t}}^{\z_{t}}F(\tz_{t}))$ in REG to eliminate it. For ROGDA, note that the proof in \citet{wang2023riemannian} heavily  relies on the holonomy effect on a geodesic quadrilateral:
 \[
 \begin{split}
    &\textstyle\|\G_{\w}^{\x}\G_{\z}^{\w}\G_{\y}^{\z}\G_{\x}^{\y}\u-\u\| \\
    =&O(1)\cdot(d(\x,\y)+d(\y,\z)+d(\w,\x))\cdot (d(\y,\z)+d(\w,\x)) 
 \end{split}
\]
 In our Lemma \ref{lem:holo}, we actually consider the case of $\w=\x$, so the quadrilateral degenerates to be a geodesic triangle. This aspect is crucial in the proof of our Lemma \ref{lem:decreasenorm}, where we aim to bound the holonomy effect using a specific geodesic edge 
$d(\z_{t+1},\tz_t)$, rather than summing two edges, mirroring the approach in Euclidean cases. Given these considerations,  demonstrating the last-iterate convergence of RCEG and ROGDA presents non-trivial challenges.

\section{Riemannian Past Extragradient}
\label{sec:rpeg}
In the Euclidean space, the Past Extragradient (PEG) method, introduced by \citet{popov1980modification}, performs iterates as in Equation \eqref{eq:epeg}. One of the advantages of PEG over EG is the halving of gradient queries in each iteration. The two inequalities presented in \citet[Lemma 3.1]{gorbunov2022last} are crucial for demonstrating the last-iterate convergence of PEG:
\begin{equation}\label{eq:pegeq0}
    \begin{aligned}
&0  \leq\left\langle F\left(\z_{t}\right)-F\left(\z_{t+1}\right), F\left(\tz_{t}\right)\right\rangle, \\
&\left\|F\left(\tz_{t}\right)-F\left(\z_{t+1}\right)\right\|^2  \leq L^2 \eta^2\left\|F\left(\tz_{t}\right)-F\left(\tz_{t-1}\right)\right\|^2 .
\end{aligned}
\end{equation}
And they can be deduced from Equation \eqref{eq:epeg}, Assumptions \ref{ass:mono} and \ref{ass:lips}.

 With insights gained from the Euclidean space, we introduce the Riemannian Past Extragradient (RPEG):
\begin{equation}\label{eq:rpeg}
    \begin{split}       
    &\tz_{t}=\expmap_{\z_{t}}\lt(-\eta\G_{\tz_{t-1}}^{\z_{t}}F(\tz_{t-1})\rt)\\
    &\z_{t+1}=\expmap_{\z_{t}}\lt(-\eta\G_{\tz_{t}}^{\z_{t}}F(\tz_{t})\rt)
    \end{split}
\end{equation}
One can readily observe that RPEG is distinct from ROGDA \citep{wang2023riemannian}:
\[
\textstyle\tz_{t+1}=\expmap_{\tz_t}\lt(-2\eta F(\tz_t)+\eta\G_{\tz_{t-1}}^{\tz_t}F(\tz_{t-1})\rt),
\]
even in the unconstrained scenario, owing to the non-linear nature of the exponential map. This distinction becomes particularly noteworthy when contrasted with the Euclidean setting.

    

Following \citet{gorbunov2022last}, we use a Lyapunov analysis argument to show the last-iterate convergence of RPEG.  Proof details of this part are deferred to Appendix \ref{app:rpeg} due to page limitations. We first establish the following lemma, which implies $d(\z_t,\z^*)\leq D$ holds for any $t\geq 0$.
\begin{lemma}\label{lem:peg}
   Under Assumptions \ref{ass:rm}, \ref{ass:lips2} and \ref{ass:grad}. For the iterates of RPEG in Equation \eqref{eq:rpeg} with 
    \[
   \textstyle \eta\leq\min\lt\{\frac{\bs}{141LD\sqrt{K_m}+32\bzt L}, \frac{1}{\sqrt{648K_mG^2}}\rt\},
    \]
    $d(\z_t,\z^*) \leq D$ holds for any $t\geq 0$.
\end{lemma}
Similar to REG, the trajectory of RPEG is also  bounded due to the following corollary.
\begin{cor}\label{cor:peg}
       Under Assumptions \ref{ass:rm},  \ref{ass:lips2} and \ref{ass:grad}. For the iterates of RPEG as in Equation \eqref{eq:rpeg}, $\z_t,\tz_t\in\D$
   holds for any $t\geq 0$ and 
       \[
  \textstyle  \eta\leq\min\lt\{\frac{\bs}{141LD\sqrt{K_m}+32\bzt L}, \frac{1}{\sqrt{648K_mG^2}}\rt\},
    \]
    where $\D$ is defined in Definition \ref{def:def}.
\end{cor}

    In \citet{gorbunov2022last}, it is discussed that for PEG, the function $\|F(\z_t)\|$ does not monotonically decrease with respect to $t$. However, they demonstrate that
\[
\begin{split}
    &\|F(\z_{t+1})\|^2+2\|F(\z_{t+1})-F(\tz_t)\|^2 \\
    \leq&\|F(\z_t)\|^2+2\|F(\z_t)-F(\tz_{t-1})\|^2
\end{split}
    \]
    holds when the step-size $\eta$ is small. We introduce a Riemannian counterpart to this proposition under the help of Lemma \ref{lem:holo2}.
    
\begin{lemma}\label{lem:pegnorm}
    With Assumptions \ref{ass:rm}, \ref{ass:mono2}, \ref{ass:lips2} and \ref{ass:grad}, the iterates of RPEG
    satisfies
    \[
\begin{split}
   \textstyle &\|F(\z_{t+1})\|^2+2\|F(\z_{t+1})-\G_{\tz_{t}}^{\z_{t+1}}F(\tz_{t})\|^2\\
    \leq&\|F(\z_{t})\|^2+2\|F(\z_{t})-\G_{\tz_{t-1}}^{\z_{t}}F(\tz_{t-1})\|^2\\
    &+\rho\|\G_{\tz_{t}}^{\z_{t}}F(\tz_{t})-\G_{\tz_{t-1}}^{\z_{t}}F(\tz_{t-1})\|^2
\end{split}
\]
for any \[
   \textstyle \eta\leq\min\lt\{\frac{\bs}{141LD\sqrt{K_m}+32\bzt L}, \frac{1}{\sqrt{648K_mG^2}}\rt\}.
    \]
    where $\rho\coloneqq \lt(\lt(24L^2+432K_mG^2+48GL\sqrt{2K_m}\rt)\eta^2-\frac{2}{3}\rt)$.
\end{lemma}

Now, we are able to establish a Lyapunov analysis for RPEG in Lemma \ref{lem:bestpeg}.
\begin{lemma}\label{lem:bestpeg}
We define
\[
\begin{split}
    &\Phi_{t}\coloneqq d(\z_{t},\z^*)^2\\
    &+\lambda t\eta^2\lt(\|F(\z_{t})\|^2+2\|F(\z_{t})-\G_{\tz_{t-1}}^{\z_{t}}F(\tz_{t-1})\|^2\rt)
\end{split}
\]
with $\lambda=\frac{\bs}{16}$.
     Under Assumptions \ref{ass:rm}, \ref{ass:mono2}, \ref{ass:lips2} and \ref{ass:grad}, the iterates of RPEG as in Equation \eqref{eq:rpeg}    
    satisfies $\Phi_{t+1}\leq\Phi_{t}$ for any
    \[
   \textstyle \eta\leq\min\lt\{\frac{\bs}{152LD\sqrt{K_m}+35\bzt L},\frac{1}{\sqrt{36L^2+648K_mG^2+72\sqrt{2K_m}GL}}\rt\}
    \]
\end{lemma}

The last-iterate convergence of RPEG is achieved by combining Lemmas \ref{lem:peg}, \ref{lem:pegnorm}, \ref{lem:bestpeg}, so the final step-size needs to satisfy all these requirements.
\begin{thm}\label{thm:peglast}
    Under Assumptions \ref{ass:rm}, \ref{ass:mono2}, \ref{ass:lips2}, \ref{ass:grad}. For the iterates of RPEG as in Equation \eqref{eq:rpeg},
    we can achieve {$O\lt(\frac{\bzt}{\sqrt{\bs^3T}}\rt)$} last-iterate convergence for monotone variational problems by choosing
    \[
   \textstyle \eta=\min\lt\{\frac{\bs}{152LD\sqrt{K_m}+35\bzt L},\frac{1}{\sqrt{36L^2+648K_mG^2+72\sqrt{2K_m}GL}}\rt\}
    \]
\end{thm}
Similar to REG, we can show RPEG also achieves $O\lt(\frac{1}{\sqrt{T}}\rt)$ last-iterate convergence and $O\lt(\frac{1}{T}\rt)$ average-iterate convergence, when applied to Riemannian convex-concave saddle point problems.

\begin{thm}\label{thm:rpeg}
Consider a Riemannian minimax optimization problem 
    \[
    \min_{\x\in\M_1}\max_{\y\in\M_2}f(\x,\y)
    \]
    where $f$ is geodesically convex-concave. Let $\M\coloneqq\M_1\times\M_2$, $\z^*\coloneqq \big(\begin{smallmatrix}
  \x^*\\
  \y^*
\end{smallmatrix}\big)$ to be the saddle point, and $\X=\B\lt(\x^*,\frac{\sqrt{2}D}{2}\rt)$, $\Y=\B\lt(\y^*,\frac{\sqrt{2}D}{2}\rt)$ to be geodesic balls. 
Under Assumptions \ref{ass:rm}, \ref{ass:mono2}, \ref{ass:lips2}, \ref{ass:grad}, if we apply RPEG in Equation \eqref{eq:rpeg} with
\[
   \textstyle \eta\leq\min\lt\{\frac{\bs}{192LD\sqrt{2K_m}+35\bzt L},\frac{1}{\sqrt{36L^2+648K_mG^2+72\sqrt{2K_m}GL}}\rt\}
    \]
 then
\[
\textstyle\max_{\y\in\Y}f(\x_T,\y)-\min_{\x\in\X}f(\x,\y_T)=\mhl{O\lt(\frac{\bzt}{\sqrt{\bs^3T}}\rt)},
\]
and
\[
\textstyle\max_{\y\in\Y}f(\bx_T,\y)-\min_{\x\in\X}f(\x,\by_T)= \mhl{O\lt(\frac{\bzt}{\bs T}\rt)},
\]
where $\bx_T=\expmap_{\bx_{T-1}}\lt(\frac{1}{T}\invexp_{\bx_{T-1}}\tx_T\rt)$ and $\by_T=\expmap_{\by_{T-1}}\lt(\frac{1}{T}\invexp_{\by_{T-1}}\ty_T\rt)$ are the ergodic averages of $\tx_t$ and $\ty_t$ for $t=1,\dots, T$.
\end{thm}

We wrap up this section with a comparison between our work and \citet{martinez2023minimax}. The work of \citet{martinez2023minimax} can achieve a faster $O\left(\frac{1}{T}\right)$ last-iterate convergence rate for Riemannian gsc-convex gsc-concave problems, but there are some key distinctions:  (i) their algorithm employs a double loop, while our algorithms are single looped and easier to implement, (ii) for gsc-convex gsc-concave optimization problems, they rely on the reduction from the strongly gsc-convex strongly gsc-concave case, and a predefined precision is required before starting the algorithm, while our algorithms are ``anytime" and can find better solutions the longer they are running, and (iii) while our last-iterate rate of 
$O\lt(\tfrac{1}{\sqrt{T}}\rt)$ is slower, it aligns with the lower bound for  $p$-SCLI algorithms in Euclidean space as established by \citet{golowich2020last,golowich2020tight}. 

\section{Conclusion}

In this study, we introduce Riemannian adaptations of the extragradient and past extragradient methods. We establish $O\lt(\frac{1}{T}\rt)$ 
 average-iterate convergence and $O\lt(\frac{1}{\sqrt{T}}\rt)$ 
 last-iterate convergence for minimax optimization on Riemannian manifolds. A cornerstone of our approach is the realization that the proof for last-iterate convergence in the Riemannian setting can be significantly simplified by transitioning to the Euclidean domain, provided the holonomy effect is carefully bounded. Looking forward, we are keen to explore achieving last-iterate convergence in constrained scenarios and pursuing $O\lt(\frac{1}{T}\rt)$  accelerated rate through the use of single-loop first-order algorithms. Other interesting open problems include how to improve the curvature-dependence of the last-iterate convergence rate and investigate whether it is possible to replace the exponential map with computationally more efficient retractions.

 \section*{Acknowledgments}
 We would like to thank five anonymous referees for their constructive comments and suggestions. JA acknowledges the AI4OPT Institute for its funding as part of NSF Award 2112533 and thanks the NSF for its support through Award IIS-1910077. MT is thankful for partial support by NSF DMS-1847802, Cullen-Peck Scholarship, and GT-Emory Humanity.AI Award.

\newpage
\bibliographystyle{plainnat}
\bibliography{main}
\newpage

\onecolumn
\aistatssupp{Extragradient Type Methods for Riemannian Variational Inequality Problems: \\
Supplementary Materials}
\section{Omitted Proof for Section \ref{sec:reg}}
\label{app:reg}
\subsection{Proof of Lemma \ref{lem:egpep}}
For simplicity, let $\a$, $\b$ and $\c$ be $F(\z_t)$, $F(\tz_t)$ and $F(\z_{t+1})$ respectively. We indeed have
    \[
    0\leq\lambda \la\a-\c,\b\ra
    \]
    and
    \[
\|\b-\c\|^2\leq L^2\eta^2\|\b-\a\|^2
\]
for any $\lambda\geq 0$. Add the above two inequalities and rearrange, we obtain
\[
L^2\eta^2\|\a\|^2+(L^2\eta^2-1)\|\b\|^2-\|\c\|^2+(2-\lambda)\la\b,\c\ra+\lt(\lambda-2L^2\eta^2\rt)\la\a,\b\ra\geq 0.
\]
Since $L^2\eta^2\leq 1$, there exists $\lambda\leq 2$ such that $\lambda-2L^2\eta^2\geq 0$. Applying Young's inequality on the cross product term, we have
\[
\begin{split}
    0\leq& L^2\eta^2\|\a\|^2+(L^2\eta^2-1)\|\b\|^2-\|\c\|^2+(2-\lambda)\la\b,\c\ra+\lt(\lambda-2L^2\eta^2\rt)\la\a,\b\ra\\
    \leq&L^2\eta^2\|\a\|^2+(L^2\eta^2-1)\|\b\|^2-\|\c\|^2+(2-\lambda)\frac{\|\b\|^2+\|\c\|^2}{2}+\lt(\lambda-2L^2\eta^2\rt)\frac{\|\a\|^2+\|\b\|^2}{2}\\
    =&\frac{\lambda}{2}\|\a\|^2-\frac{\lambda}{2}\|\c\|^2.
\end{split}
\]
Hence, $\|\c\|\leq\|\a\|$, which means
\[
\|F(\z_{t+1})\|\leq \|F(\z_t)\|
\]
as asserted.
\subsection{Auxillary Lemmas on the Iterates of REG}
To demonstrate the $O\lt(\frac{1}{\sqrt{T}}\rt)$  best-iterate convergence of REG, the succeeding three lemmas prove to be instrumental in this regard. Lemma~\ref{lem:eggrad1} indicates that $d(\tz_t,\z_{t+1})=O(\eta^2)$, while Lemma~\ref{lem:eggrad2} elucidates the relationship between $\|F(\z_t)\|$ and $\|F(\tz_t)\|$. Lemma \ref{lem:egtraj} is a helpful lemma for proving that all iterates of REG remain bounded.

 
\begin{lemma}\label{lem:eggrad1}
  Under Assumptions \ref{ass:rm}, \ref{ass:lips2}.  For the iterates of REG in Equation \eqref{eq:reg},
    suppose $\eta$ is chosen to ensure $\max\{d(\z_t,\tz_t),d(\z_t,\z_{t+1})\}\leq\frac{1}{\sqrt{K_m}}$,
    then we have
    \[
    d(\tz_{t},\z_{t+1})\leq 2 L\eta^2\|F(\z_{t})\|.
    \]
\end{lemma}
\begin{proof}
    We have
    \[
    \begin{split}
        &d(\tz_{t},\z_{t+1})\\
        \leq& 2\|\invexp_{\z_{t}}\tz_{t}-\invexp_{\z_{t}}\z_{t+1}\|=2\eta\cdot\|F(\z_{t})-\G_{\tz_{t}}^{\z_{t}}F(\tz_{t})\|\\
        \leq& 2\eta L d(\z_{t},\tz_{t})=2L\eta^2\|F(\z_{t})\|.
    \end{split}
    \]
    where the first inequality is due to Lemma \ref{lem:trig} and the second one is due to Assumption \ref{ass:lips2}.
\end{proof}

\begin{lemma}\label{lem:eggrad2}
Under Assumptions \ref{ass:rm}, \ref{ass:lips2}.
    For the iterates of REG in Equation \eqref{eq:reg},
    we have
    \[
    (1-L\eta)\|F(\z_{t})\|\leq\|F(\tz_{t})\|\leq(1+L\eta)\|F(\z_{t})\|.
    \]
\end{lemma}

\begin{proof}
Due to the triangle inequality,
\[
\|F(\z_t)\|+\|\G_{\z_t}^{\tz_t}F(\z_t)-F(\tz_t)\|\leq\|F(\tz_t)\|=\|\G_{\z_t}^{\tz_t}F(\z_t)-\G_{\z_t}^{\tz_t}F(\z_t)+F(\tz_t)\|\leq \|F(\z_t)\|+\|\G_{\z_t}^{\tz_t}F(\z_t)-F(\tz_t)\|.
\]
According to Assumption \ref{ass:lips2}, we have
\[
\|\G_{\z_t}^{\tz_t}F(\z_t)-F(\tz_t)\|\leq L\cdot d(\z_t,\tz_t)=\eta L\|F(\z_t)\|.
\]
By combining the two aforementioned inequalities, the proof of the lemma is complete.
\end{proof}



    
\begin{lemma}\label{lem:egtraj}
Considering the iterates of REG as given in Equation~\eqref{eq:reg} and with a step-size of $\eta\leq\frac{1}{9L}$, if we assume that $d(\z_t,\z^*)\leq D$ and Assumptions~\ref{ass:rm} and~\ref{ass:lips2} are valid, then the following results hold:
    \[
    d(\z_{t+1},\z^*)\leq\frac{91D}{81}
    \]
    and
    \[
    d(\tz_t,\z^*)\leq\frac{10D}{9}.
    \]
    We can also obtain $\z_t,\z_{t+1},\tz_t\in\D$ holds where $\D$ is defined in Definition \ref{def:def}.
\end{lemma}

\begin{proof}
    This lemma can be proved via a combination of the triangle inequality, Lemma \ref{lem:eggrad2} and Assumption \ref{ass:lips2}. For $d(\z_{t+1},\z^*)$:
    \[
    \begin{split}
        &d(\z_{t+1},\z^*)\leq d(\z_t,\z^*)+d(\z_{t+1},\z_t)\\
        \leq& D+\eta\|F(\tz_t)\|\leq D+\eta(1+\eta L)\|F(\z_t)\|\\
        \leq &D+\frac{10}{9}\eta DL\leq D+\frac{10}{9}\frac{1}{9L}DL=\frac{91D}{81}.
    \end{split}
    \]   
    We can bound $d(\tz_t,\z^*)$ in a similar way:
    \[
    \begin{split}
        &d(\tz_t,\z^*)\leq d(\z_t,\z^*)+d(\tz_t,\z_t)\\
        \leq& D+\eta\|F(\z_t)\|\leq D+\eta DL\leq D+\frac{1}{9L}DL=\frac{10D}{9}.
    \end{split}
    \]
\end{proof}


\subsection{Proof of Lemma \ref{lem:egbest}}
\begin{proof}
To confirm that $d(\z_t,\z^*)\leq D$ for any $t\geq 0$, we employ an induction approach. The base case $d(\z_0,\z^*)\leq D$ is straightforward. Assume $d(\z_t,\z^*)\leq D$, we proceed to establish that $d(\z_{t+1},\z^*)\leq D$.

By Riemannian cosine law Lemma \ref{lem:cos2}, we have
\begin{equation}
    d(\z_{t+1},\z^*)^2-d(\z_{t},\z^*)^2\leq 2\la\invexp_{\z_{t+1}}\z_{t},\invexp_{\z_{t+1}}\z^*\ra-\sigma(K,d(\z_t,\z_{t+1})+\min\lt\{d(\z_t,\z^*),d(\z_{t+1},\z^*)\rt\}) \cdot d(\z_{t},\z_{t+1})^2.
\end{equation}
Based on the monotonicity of $\sigma(K,\cdot)$ and $\eta\leq \frac{\bs}{8\bzt L+\bs L}\leq\frac{1}{9L}$, we have
\begin{equation}\label{eq:egsigma1}
    \begin{split}
        &-\sigma(K,d(\z_t,\z_{t+1})+\min\lt\{d(\z_t,\z^*),d(\z_{t+1},\z^*)\rt\})\\
        \leq&-\sigma(K,d(\z_t,\z^*)+d(\z_t,\z_{t+1}))\\
        \leq &-\sigma\lt(K,\frac{91D}{81}\rt)=-\bs,
    \end{split}
\end{equation}
The second inequality is based on the condition
$d(\z_t,\z^*)+d(\z_{t+1},\z_t)\leq \frac{91D}{81}$, as established in the proof of Lemma \ref{lem:egtraj}.
Thus, we can deduce
\begin{equation}\label{eq:egbest1}
    d(\z_{t+1},\z^*)^2-d(\z_{t},\z^*)^2\leq 2\la\invexp_{\z_{t+1}}\z_{t},\invexp_{\z_{t+1}}\z^*\ra-\bs \cdot d(\z_{t},\z_{t+1})^2.
\end{equation}
by combining the above two inequalities.


We establish an upper bound for 
$2\la\invexp_{\z_{t+1}}\z_{t},\invexp_{\z_{t+1}}\z^*\ra$ as follows:
\begin{equation}\label{eq:egbest2}
    \begin{split} &2\la\invexp_{\z_{t+1}}\z_{t},\invexp_{\z_{t+1}}\z^*\ra=2\eta\la\G_{\z_{t}}^{\z_{t+1}}\G_{\tz_{t}}^{\z_{t}}F(\tz_{t}),\invexp_{\z_{t+1}}\z^*\ra\\
    =&2\eta\la\G_{\z_{t}}^{\z_{t+1}}\G_{\tz_{t}}^{\z_{t}}F(\tz_{t})-\G_{\tz_{t}}^{\z_{t+1}}F(\tz_{t}),\invexp_{\z_{t+1}}\z^*\ra+2\eta\la \G_{\tz_{t}}^{\z_{t+1}}F(\tz_{t}),\invexp_{\z_{t+1}}\z^*\ra\\
{\leq}&2\eta\|\G_{\z_{t}}^{\z_{t+1}}\G_{\tz_{t}}^{\z_{t}}F(\tz_{t})-\G_{\tz_{t}}^{\z_{t+1}}F(\tz_{t})\|\cdot d(\z_{t+1},\z^*)+2\eta\la \G_{\tz_{t}}^{\z_{t+1}}F(\tz_{t}),\invexp_{\z_{t+1}}\z^*\ra\\
    \end{split}
\end{equation}
The term $\|\G_{\z_{t}}^{\z_{t+1}}\G_{\tz_{t}}^{\z_{t}}F(\tz_{t})-\G_{\tz_{t}}^{\z_{t+1}}F(\tz_{t})\|$ corresponds to the geometric distortion due to the holonomy effect and can be addressed by Lemma \ref{lem:holo}. Specifically,
\begin{equation}\label{eq:egbest3}
\begin{split}
    &\|\G_{\z_{t}}^{\z_{t+1}}\G_{\tz_{t}}^{\z_{t}}F(\tz_{t})-\G_{\tz_{t}}^{\z_{t+1}}F(\tz_{t})\| \\
    \leq&36K_m\|F(\tz_t)\|\cdot\min\{d(\z_{t+1},\tz_t)+d(\z_t,\z_{t+1}),d(\z_{t+1},\tz_t)+d(\z_t,\tz_t)\}\cdot d(\z_{t+1},\tz_t)\\
    \leq &36c\sqrt{K_m}\|F(\tz_t)\|\cdot d(\tz_t,\z_{t+1})\\
    \leq &72c\sqrt{K_m}L\eta^2\|F(\tz_t)\|\cdot\|F(\z_t)\|
\end{split}
\end{equation}
where the second inequality is due to 
\begin{equation}\label{eq:egholo}
    \begin{split}
   &\min\{d(\z_{t+1},\tz_t)+d(\z_t,\z_{t+1}),d(\z_{t+1},\tz_t)+d(\z_t,\tz_t)\}\\
   \leq& d(\z_{t+1},\z_t)+2d(\z_t,\tz_t)=\eta\|F(\tz_t)\|+2\eta\|F(\z_t)\|\\
   \leq& 3\eta G\leq\frac{3}{\sqrt{306K_m}}\leq\frac{1}{1152^{\frac{1}{4}}\sqrt{K_m}}\coloneqq \frac{c}{\sqrt{K_m}},
\end{split}
\end{equation}
 and the third one follows from Lemma \ref{lem:eggrad1}. Note that the second inequality of Equation \eqref{eq:egholo} is due to Lemma \ref{lem:egtraj} and Assumption \ref{ass:grad}. Now it remains to bound $2\eta\la \G_{\tz_{t}}^{\z_{t+1}}F(\tz_{t}),\invexp_{\z_{t+1}}\z^*\ra$:
\begin{equation}\label{eq:egbest4}
    \begin{split}
        &2\eta\la \G_{\tz_{t}}^{\z_{t+1}}F(\tz_{t}),\invexp_{\z_{t+1}}\z^*\ra=2\eta\la F(\tz_{t}),\G_{\z_{t+1}}^{\tz_{t}}\invexp_{\z_{t+1}}\z^*\ra \\
        =&2\eta\la F(\tz_{t}),\G_{\z_{t+1}}^{\tz_{t}}\invexp_{\z_{t+1}}\z^*-\invexp_{\tz_{t}}\z^*\ra +2\eta\la F(\tz_{t}),\invexp_{\tz_{t}}\z^*\ra\\
        \leq&2\eta\bzt\|F(\tz_{t})\|\cdot d(\tz_{t},\z_{t+1})+0\\
        \leq &4\eta\bzt L\eta^2\|F(\tz_{t})\|\cdot\|F(\z_{t})\|=4\eta^3\bzt L\|F(\z_{t})\|\cdot \|F(\tz_{t})\|.
    \end{split}
\end{equation}
where the first inequality is by Lemma \ref{lem:hessian}, the monotonicity of $\zeta(\k,\cdot)$ and
\begin{equation}\label{eq:egzeta}
    \begin{split}
    &d(\tz_t,\z_{t+1})+\min\lt\{d(\tz_t,\z^*),d(\z_{t+1},\z^*)\rt\}\\
    \leq& d(\tz_t,\z_t)+d(\z_t,\z_{t+1})+d(\z_t,\z^*)+d(\tz_t,\z_t)\\
    \leq&2\eta\|F(\z_t)\|+\eta\|F(\tz_t)\|+D\\
    \leq&\lt(2\eta+\eta(1+\eta L)\rt)\|F(\z_t)\|+D\\
    \leq&\frac{28}{9}\eta\cdot DL+D\\
    \leq&\frac{28D}{81}+D\leq \frac{7D}{5}.
\end{split}
\end{equation}
And the second inequality of Equation \eqref{eq:egbest4} is due to Lemma \ref{lem:eggrad1}.
Combining Equations \eqref{eq:egbest1}, \eqref{eq:egbest2}, \eqref{eq:egbest3} and \eqref{eq:egbest4}, we have
\begin{equation}\label{eq:egbest5}
    \begin{split}
        &d(\z_{t+1},\z^*)^2-d(\z_{t},\z^*)^2\leq 2\la\invexp_{\z_{t+1}}\z_{t},\invexp_{\z_{t+1}}\z^*\ra-\bs d(\z_{t},\z_{t+1})^2\\
    \leq& \eta^3\lt(144c\sqrt{K_m}\frac{91}{81}DL+4\bzt L\rt)\|F(\z_{t})\|\cdot\|F(\tz_{t})\|-\bs\eta^2\|F(\tz_{t})\|^2\\
    \leq& \eta^2\lt(\frac{(144c\sqrt{K_m}\frac{91}{81}DL+4\bzt L)\eta}{1-L\eta}-\bs\rt)\|F(\tz_{t})\|^2\\
    \leq&-\frac{\bs\eta^2}{2}\|F(\tz_{t})\|^2.
    \end{split}
\end{equation}
where we use Lemma \ref{lem:eggrad2}, Lemma \ref{lem:egtraj}, the fact that $c=\frac{1}{1152^{\frac{1}{4}}}$ and $\eta\leq \frac{\bs}{56\sqrt{K_m}DL+8\bzt L+\bs L}\leq\frac{\bs}{288c\sqrt{K_m}\frac{91}{81}DL+8\bzt L+L\bs}$. By Equation \eqref{eq:egbest5}, we know $d(\z_{t+1},\z^*)\leq d(\z_t,\z^*)\leq D$, and by induction, $d(\z_t,\z^*)\leq D$ holds for any $t\geq 0$.

Also, since $\eta\leq\frac{\bs}{(8\bzt+\bs)L}$, we sum over Equation \eqref{eq:egbest5} from $t=1$ to $T$ to obtain:

\[
\begin{split}
    &\sum_{t=1}^T\|F(\z_{t})\|^2\leq\sum_{t=1}^T\frac{1}{(1-L\eta)^2}\|F(\tz_{t})\|^2\leq\sum_{t=1}^T\frac{(8\bzt+\bs)^2}{(8\bzt)^2}\cdot\|F(\tz_{t})\|^2\\
    \leq &\frac{2}{\bs\eta^2}\lt(\frac{8\bzt+\bs}{8\bzt}\rt)^2\cdot d(\z_0,\z^*)^2=\mhl{O\lt(\frac{\bzt^2}{\bs^3}\rt)}.
\end{split}
\]
Thus,
\[
T\cdot\min_{t'\in[T]}\|F(\z_{t'})\|^2\leq\sum_{t=1}^T\|F(\z_{t})\|^2=\mhl{O\lt(\frac{\bzt^2}{\bs^3}\rt)},
\]
and there exists $t'\in[T]$, such that
\[
\|F(\z_{t'})\|=\mhl{O\lt(\frac{\bzt}{\sqrt{\bs^3T}}\rt)}.
\]
\end{proof}
\subsection{Proof of Lemma \ref{lem:decreasenorm}}
\begin{proof}
It is noteworthy that the chosen step-size $\eta$ satisfies the requirement of Corollary \ref{cor:eg}, thus $\z_t,\tz_t\in\D$ for all $t\geq 0$. This observation is pivotal since Assumptions~\ref{ass:mono2} and~\ref{ass:grad} are only applicable to $\D$.


We can directly show an analog of the first inequality in Equation \eqref{eq:peg} by
\begin{equation}\label{eq:egpep1}
    0\leq \la\G_{\z_{t+1}}^{\z_{t}}F(\z_{t+1})-F(\z_{t}),\invexp_{\z_{t}}\z_{t+1}\ra=\la F(\z_{t})-\G_{\z_{t+1}}^{\z_{t}}F(\z_{t+1}),\eta\G_{\tz_{t}}^{\z_{t}}F(\tz_{t})\ra,
\end{equation}
where the inequality is due to Assumption \ref{ass:mono2}, while the equality follows from Equation \eqref{eq:reg}. Achieving an analog of the second inequality in Equation \eqref{eq:peg} is more complicated. We first show
\begin{equation}\label{eq:egpep2}
    \begin{split}
        &\|\G_{\tz_{t}}^{\z_{t}}F(\tz_{t})-\G_{\z_{t+1}}^{\z_{t}}F(\z_{t+1})\|^2\\
        =&\|\G_{\tz_{t}}^{\z_{t}}F(\tz_{t})-\G_{\tz_{t}}^{\z_{t}}\G_{\z_{t+1}}^{\tz_{t}}F(\z_{t+1})+\G_{\tz_{t}}^{\z_{t}}\G_{\z_{t+1}}^{\tz_{t}}F(\z_{t+1})-\G_{\z_{t+1}}^{\z_{t}}F(\z_{t+1})\|^2\\
       {\leq} &2\|\G_{\tz_{t}}^{\z_{t}}F(\tz_{t})-\G_{\tz_{t}}^{\z_{t}}\G_{\z_{t+1}}^{\tz_{t}}F(\z_{t+1})\|^2
       +2\|\G_{\tz_{t}}^{\z_{t}}\G_{\z_{t+1}}^{\tz_{t}}F(\z_{t+1})-\G_{\z_{t+1}}^{\z_{t}}F(\z_{t+1})\|^2      
    \end{split}
\end{equation}
where the inequality is due to $\|\a+\b\|^2\leq 2(\|\a\|^2+\|\b\|^2)$.

Applying Lemma \ref{lem:holo} with $\x=\z_t$, $\y=\z_{t+1}$, $\z=\tz_t$ and $\G_{\x}^{\y}\u=F(\z_{t+1})$, then we can easily verify
\[
\begin{split}
    \|\G_{\z}^{\x}\G_{\y}^{\z}\G_{\x}^{\y}\u-\u\|=&\|\G_{\tz_{t}}^{\z_{t}}\G_{\z_{t+1}}^{\tz_{t}}\G_{\z_{t}}^{\z_{t+1}}\G_{\z_{t+1}}^{\z_{t}}F(\z_{t+1})-\G_{\z_{t+1}}^{\z_{t}}F(\z_{t+1})\|\\    =&\|\G_{\tz_{t}}^{\z_{t}}\G_{\z_{t+1}}^{\tz_{t}}F(\z_{t+1})-\G_{\z_{t+1}}^{\z_{t}}F(\z_{t+1})\|
\end{split}
\]
 and notice
\[
3\eta G\leq\frac{3}{\sqrt{306K_m}}\leq\frac{1}{1152^{\frac{1}{4}}\sqrt{K_m}}\coloneqq\frac{c}{\sqrt{K_m}}
\]
by $\eta\leq\frac{1}{\sqrt{8L^2+306K_mG^2}}$. Thus, by the triangle inequality,
\[
\min\{d(\z_t,\z_{t+1})+d(\z_{t+1},\tz_t),d(\z_t,\tz_t)+d(\z_{t+1},\tz_t)\}\leq 3\eta G\leq\frac{c}{\sqrt{K_m}},
\]
and we have
\begin{equation}\label{eq:egpep3}
    \begin{split}  
    &\|\G_{\tz_{t}}^{\z_{t}}\G_{\z_{t+1}}^{\tz_{t}}F(\z_{t+1})-\G_{\z_{t+1}}^{\z_{t}}F(\z_{t+1})\|\\  
\leq& 36K_m\|F(\z_{t+1})\|\cdot \min\{d(\z_t,\z_{t+1})+d(\z_{t+1},\tz_t),d(\z_t,\tz_t)+d(\z_{t+1},\tz_t)\}\cdot d(\z_{t+1},\tz_t)\\
\leq & 36\cdot c\sqrt{K_m}G\cdot d(\tz_t,\z_{t+1}),
\end{split}
\end{equation}
where we use Corollary \ref{cor:eg} to show $\z_{t+1}\in\D$, and thus $\|F(\z_{t+1})\|\leq G$ by Assumption \ref{ass:grad}.


Combining Equations \eqref{eq:egpep2} and \eqref{eq:egpep3} yields
\begin{equation}\label{eq:egpep4}
    \begin{split}
        &\|\G_{\tz_{t}}^{\z_{t}}F(\tz_{t})-\G_{\z_{t+1}}^{\z_{t}}F(\z_{t+1})\|^2\\
        \stackrel{(1)}{\leq} &2\|\G_{\tz_{t}}^{\z_{t}}F(\tz_{t})-\G_{\tz_{t}}^{\z_{t}}\G_{\z_{t+1}}^{\tz_{t}}F(\z_{t+1})\|^2+2\|\G_{\tz_{t}}^{\z_{t}}\G_{\z_{t+1}}^{\tz_{t}}F(\z_{t+1})-\G_{\z_{t+1}}^{\z_{t}}F(\z_{t+1})\|^2\\
        \stackrel{(2)}{\leq}&(2L^2+2592 c^2K_mG^2)d(\tz_{t},\z_{t+1})^2\\
        \stackrel{(3)}{\leq}&(2L^2+2592c^2 K_mG^2)\cdot 4\|\invexp_{\z_{t}}\tz_{t}-\invexp_{\z_{t}}\z_{t+1}\|^2\\
        =&(8L^2+10368 c^2 K_mG^2)\eta^2\|F(\z_{t})-\G_{\tz_{t}}^{\z_{t}}F(\tz_{t})\|^2\\
        \stackrel{(4)}{\leq}&(8L^2+306  K_mG^2)\eta^2\|F(\z_{t})-\G_{\tz_{t}}^{\z_{t}}F(\tz_{t})\|^2
    \end{split}
\end{equation}
where the second inequality is by Assumption \ref{ass:lips2}, the third is by Lemma \ref{lem:trig}, and for the last inequality, we use the fact that $10368c^2=\frac{10368}{\sqrt{1152}}\leq 306$.

Combining Equations \eqref{eq:egpep1} and \eqref{eq:egpep4}, we have
\begin{equation}
    \begin{split}
        &0\leq \la F(\z_{t})-\G_{\z_{t+1}}^{\z_{t}}F(\z_{t+1}),\G_{\tz_{t}}^{\z_{t}}F(\tz_{t})\ra\\
        &\|\G_{\tz_{t}}^{\z_{t}}F(\tz_{t})-\G_{\z_{t+1}}^{\z_{t}}F(\z_{t+1})\|^2\leq (8L^2+306  K_mG^2)\eta^2\|F(\z_{t})-\G_{\tz_{t}}^{\z_{t}}F(\tz_{t})\|^2.
    \end{split}
\end{equation}
Since $\eta$ satisfying $(8L^2+306 K_mG^2)\eta^2\leq 1$, we can apply Lemma \ref{lem:egpep} with $\a=F(\z_{t})$, $\b=\G_{\tz_{t}}^{\z_{t}}F(\tz_{t})$ and $\c=\G_{\z_{t+1}}^{\z_{t}}F(\z_{t+1})$ to obtain 
\[
\|F(\z_{t+1})\|=\|\G_{\z_{t+1}}^{\z_{t}}F(\z_{t+1})\|\leq \|F(\z_{t})\|,
\]
where the equality is due to the parallel transport preserves the vector norm. We note that $F(\z_{t}),\G_{\tz_{t}}^{\z_{t}}F(\tz_{t}),\G_{\z_{t+1}}^{\z_{t}}F(\z_{t+1})$ are tangent vectors in the same tangent space $T_{\z_t}\M$.
\end{proof}
\subsection{Proof of Theorem \ref{thm:eglast}}
\begin{proof}
The theorem can be proved by directly combining  Lemma \ref{lem:egbest} and Lemma \ref{lem:decreasenorm}.
\end{proof}

\subsection{Proof of Theorem \ref{thm:reg}}
 To establish the average-iterate convergence, the following lemma proves beneficial.
\begin{lemma}\label{lem:egprimal}
    Under Assumptions \ref{ass:rm},  \ref{ass:lips2} and \ref{ass:grad}. For the iterates of REG as in Equation \eqref{eq:reg} with 
    \[
    \eta\leq\min\lt\{\frac{1}{\sqrt{8L^2+306K_mG^2}},\frac{\bs}{56\sqrt{K_m}DL+8\bzt L+\bs L} \rt\},
    \]
    we have
    \[
    d(\z_{t+1},\z)^2-d(\z_t,\z)^2\leq 2\eta\la F(\tz_t),\invexp_{\tz_t}\z\ra.
    \]
    holds for any $t\geq 0$ and $\z\in\B(\z^*,D)$, where $\B(\z^*,D)$ denotes the geodesic ball with center $\z^*$ and radius $D$.
\end{lemma}
\begin{proof}
The proof is similar to that of Lemma \ref{lem:egbest}.
    Combining Equations \eqref{eq:egbest1}, \eqref{eq:egbest2}, \eqref{eq:egbest3} and \eqref{eq:egbest4} but replacing $\z^*$ with $\z$, we have
    \begin{equation}\label{eq:egprim}
        \begin{split}
            &d(\z_{t+1},\z)^2-d(\z_{t},\z)^2\leq 2\la\invexp_{\z_{t+1}}\z_{t},\invexp_{\z_{t+1}}\z\ra-\bs d(\z_{t},\z_{t+1})^2\\
    \leq& \eta^3\lt(\frac{144}{1152^{\frac{1}{4}}}\sqrt{K_m}\cdot d(\z_{t+1},\z)L+4\bzt L\rt)\|F(\z_{t})\|\cdot\|F(\tz_{t})\|-\bs\eta^2\|F(\tz_{t})\|^2+2\eta\la F(\tz_t),\invexp_{\tz_t}\z\ra\\
    \leq& \eta^3\lt(\frac{144}{1152^{\frac{1}{4}}}\sqrt{K_m}\cdot 2DL+4\bzt L\rt)\|F(\z_{t})\|\cdot\|F(\tz_{t})\|-\bs\eta^2\|F(\tz_{t})\|^2+2\eta\la F(\tz_t),\invexp_{\tz_t}\z\ra\\
    \leq& \frac{\eta^3}{1-L\eta}\lt(\frac{144}{1152^{\frac{1}{4}}}\sqrt{K_m}\cdot 2DL+4\bzt L\rt)\cdot\|F(\tz_{t})\|^2-\bs\eta^2\|F(\tz_{t})\|^2+2\eta\la F(\tz_t),\invexp_{\tz_t}\z\ra,\\
        \end{split}
    \end{equation}
    where the third inequality is due to Lemma \ref{lem:egbest} and $d(\z_{t+1},\z)\leq d(\z_{t+1},\z^*)+d(\z^*,\z)\leq 2D$, and the last inequality follows from Lemma \ref{lem:eggrad2}. Now, we can pick up $\eta$ to ensure 
    \[
    \frac{\eta^3}{1-L\eta}\lt(\frac{144}{1152^{\frac{1}{4}}}\sqrt{K_m}\cdot 2DL+4\bzt L\rt)\cdot\|F(\tz_{t})\|^2-\bs\eta^2\|F(\tz_{t})\|^2\leq 0,
    \]
    which means:
    \[
    \eta\leq\frac{\bs}{\frac{288}{1152^{\frac{1}{4}}}\sqrt{K_m}DL+4\bzt L+\bs L}\approx\frac{\bs}{49.43\sqrt{K_m}DL+4\bzt L+\bs L}.
    \]
    Therefore,
    \[
    \eta\leq\min\lt\{\frac{1}{\sqrt{8L^2+306K_mG^2}},\frac{\bs}{56\sqrt{K_m}DL+8\bzt L+\bs L} \rt\}\leq \frac{\bs}{\frac{288}{1152^{\frac{1}{4}}}\sqrt{K_m}DL+4\bzt L+\bs L},
    \]
    which concludes the proof.
\end{proof}
Now we are able to provide the proof of Theorem \ref{thm:reg}.
\begin{proof}
We denote $\z=\big(\begin{smallmatrix}
    \x\\ \y\end{smallmatrix}\big)$ and $\Z=\B(\z^*,D)$ for convenience. To show the last-iterate convergence, 
\[
\begin{split}
    &\max_{\y\in\Y}f(\x_T,\y)-\min_{\x\in\X}f(\x,\y_T)\leq\max_{\y\in\Y}\la\nabla_{\y}f(\x_T,\y_T),\invexp_{\y_T}\y\ra+\max_{\x\in\X}\la\nabla_{\x}f(\x_T,\y_T),-\invexp_{\x_T}\x\ra\\
    =&\max_{\z\in\Z}\la -F(\z_T),\invexp_{\z_T}\z\ra
    \leq \|F(\z_T)\|\cdot\max_{\z\in\Z}d(\z_T,\z)\leq \|F(\z_T)\|\cdot\lt(d(\z_T,\z^*)+\max_{\z\in\Z}d(\z^*,\z)\rt)\\
    \leq& \|F(\z_T)\|\cdot 2D=\mhl{O\lt(\frac{\bzt}{\sqrt{\bs^3T}}\rt)}.
\end{split}
\]
where the last equality is due to Theorem \ref{thm:eglast}. For the average-iterate convergence, 
\[
\begin{split}
    f(\bx_T,\y)-f(\x,\by_T)\stackrel{(1)}{\leq}&\frac{1}{T}\lt(\sumt f(\tx_t,\y)-\sumt f(\x,\ty_t)\rt)\\
    \stackrel{(2)}{\leq}&\frac{1}{T}\sumt\la-F(\tz_t),\invexp_{\tz_t}\z\ra\\
    \stackrel{(3)}{\leq}&\frac{1}{2\eta T}\sumt\lt(d(\z_t,\z)^2-d(\z_{t+1},\z)^2\rt)\\
    \stackrel{(4)}{\leq}&\frac{d(\z_0,\z)^2}{2\eta T}\leq \frac{(2D)^2}{2\eta T}=\frac{2D^2}{\eta T}=\mhl{O\lt(\frac{\bzt}{\bs T}\rt)},
\end{split}
\]
where the first inequality is due to a nested application of Jensen's inequality for gsc-convex functions:
\[
f(\bx_t,\y)\leq\frac{1}{t}f(\tx_t,\y)+\frac{t-1}{t}f(\bx_{t-1},\y),
\] the second is by the gsc-convexity, and the third comes from Lemma \ref{lem:egprimal}.

\end{proof}

\subsection{Challenge for Establishing the Last-iterate Convergence of RCEG}
\label{app:rceg}
We briefly touch upon our decision to omit a discussion on the last-iterate convergence of RCEG, as proposed by \citet{zhang2022minimax}. This variant introduces a correction term to ensure metric compatibility:
\[
    \begin{split}
        &\tz_{t}=\expmap_{\z_{t}}(-\eta F(\z_{t}))\\
        &\z_{t+1}=\expmap_{\tz_{t}}(-\eta F(\tz_{t})+\invexp_{\tz_{t}}(\z_{t})).
    \end{split}
    \]
A keen reader may wonder about the behavior of RCEG when we aim to derive an equivalent of Equation~\eqref{eq:peg}. For RCEG, crafting a counterpart to the first inequality in Equation~\eqref{eq:peg} appears challenging.

By Lemma \ref{lem:hessian}, we have
    \[
    \begin{split}
        0\leq&\la F(\z_t)-\G_{\z_{t+1}}^{\z_t}F(\z_{t+1}),-\invexp_{\z_t}\z_{t+1}\ra\\
        =&\la F(\z_t)-\G_{\z_{t+1}}^{\z_t}F(\z_{t+1}),\G_{\tz_t}^{\z_t}(\invexp_{\tz_t}\z_t-\invexp_{\tz_t}\z_{t+1})\ra\\
        &+\la F(\z_t)-\G_{\z_{t+1}}^{\z_t}F(\z_{t+1}),-\invexp_{\z_t}{\z_{t+1}}-(\G_{\tz_t}^{\z_t}(\invexp_{\tz_t}\z_t-\invexp_{\tz_t}\z_{t+1}))\ra\\
        \leq&\eta \la F(\z_t)-\G_{\z_{t+1}}^{\z_t}F(\z_{t+1}),\G_{\tz_t}^{\z_t}F(\tz_t)\ra +L\cdot \max\{\zeta(\k,\tau)-1,1-\sigma(K,\tau)\}\cdot d(\z_t,\z_{t+1})\cdot d(\z_t,\tz_t)\\
    \end{split}
    \]
    where $\tau=d(\z_t,\tz_t)+\min\lt\{d(\z_t,\z_{t+1}),d(\tz_t,\z_{t+1})\rt\}$.    
We observe that the distortion term, proportional to $d(\z_t,\z_{t+1})\cdot d(\z_t,\tz_t)$, poses challenges in establishing bounds. In contrast, for REG, the geometric distortion due to the holonomy effect can be handled  with the aid of Lemma \ref{lem:holo}. Hence, in this study, we predominantly focus on REG, demonstrating that it indeed achieves $O\lt(\frac{1}{\sqrt{T}}\rt)$ last-iterate convergence.
\section{Omitted Proof for Section \ref{sec:rpeg}}
\label{app:rpeg}
\subsection{Auxillary Lemmas on the Iterates of RPEG}
Lemma \ref{lem:peggrad3} confirms that 
$d(\tz_t,\z_{t+1})=O(\eta^2)$, whereas Lemmas \ref{lem:peggrad1} and \ref{lem:peggrad2} delineate the relationship between $\|F(\z_t)\|$, $\|F(\tz_t)\|$ and $\|F(\tz_{t+1})\|$. Lemma \ref{lem:pegtraj} demonstrates that if  $\z_t$ is bounded, then $\z_{t+1},\tz_t,\tz_{t-1}$ are also bounded.
\begin{lemma}\label{lem:peggrad3}
    For the iterates of RPEG as in Equation \eqref{eq:rpeg}
     with $\eta\leq\frac{1}{G\sqrt{K_m}}$, we have
    \[
    d(\tz_t,\z_{t+1})\leq 2L\eta^2(2\|F(\tz_{t-1})\|+\|F(\tz_{t-2})\|).
    \]
\end{lemma}
\begin{proof}
    We begin with Lemma \ref{lem:trig},
    \[
    \begin{split}
        &d(\tz_t,\z_{t+1})\leq 2\|\invexp_{\z_t}\tz_t-\invexp_{\z_t}\z_{t+1}\|=2\eta\|\G_{\tz_{t-1}}^{\z_t}F(\tz_{t-1})-\G_{\tz_t}^{\z_t}F(\tz_t)\|\\
        =&2\eta\|\G_{\tz_{t-1}}^{\z_t}F(\tz_{t-1})-F(\z_t)+F(\z_t)-\G_{\tz_t}^{\z_t}F(\tz_t)\|\\
        \leq&2\eta L\cdot d(\z_t,\tz_{t-1})+2\eta L d(\z_t,\tz_t)\\
        \leq &2\eta L(d(\z_t,\z_{t-1})+d(\tz_{t-1},\z_{t-1})+d(\z_t,\tz_t))\\
        =&2\eta^2L(\|F(\tz_{t-1})\|+\|F(\tz_{t-2})\|+\|F(\tz_{t-1})\|)\\
        =&2\eta^2L(2\|F(\tz_{t-1})\|+\|F(\tz_{t-2})\|).
    \end{split}
    \]
\end{proof}
\begin{lemma}\citep{chavdarova2021last}\label{lem:peggrad1}
    Suppose $\eta\leq\frac{1}{8L}$, then 
    \[
    \frac{1}{2}\leq\frac{\|F(\tz_{t+1})\|}{\|F(\tz_{t})\|}\leq\frac{3}{2}
    \]
    holds for RPEG in Equation \eqref{eq:rpeg}.
\end{lemma}
In \citet{chavdarova2021last}, Lemma \ref{lem:peggrad1} has been established for PEG. However, given that the primary argument hinges on the triangle inequality and the Lipschitz continuity of 
$F$, extending the proof to the manifold setting is straightforward.

\begin{lemma}\label{lem:peggrad2}
    Suppose $\eta\leq\frac{1}{8L}$, for the iterates of RPEG as in Equation \eqref{eq:rpeg},
    we have
    \[    
    (1-2L\eta)\|F(\tz_{t})\|\leq\|F(\z_{t})\|\leq(1+2L\eta)\|F(\tz_{t})\|.
    \]
\end{lemma}
\begin{proof}
    The proof is immediate by the triangle inequality, Lemma \ref{lem:peggrad1} and Assumption \ref{ass:lips2}. First, we have
    \[
    \begin{split}
        \|F(\z_t)\|\leq&\|F(\tz_t)\|+\|F(\z_t)-\G_{\tz_t}^{\z_t}F(\tz_t)\|\\
        =&\|F(\tz_t)\|+L\cdot d(\z_t,\tz_t)=\|F(\tz_t)\|+L\eta\|F(\tz_{t-1})\|\\
        \leq&(1+2\eta L)\|F(\tz_t)\|.        
    \end{split}
    \]
    Next, we demonstrate that:
    \[
    \begin{split}
        \|F(\z_t)\|\geq&\|F(\tz_t)\|-\|F(\z_t)-\G_{\tz_t}^{\z_t}F(\tz_t)\|\\
        =&\|F(\tz_t)\|-L\cdot d(\z_t,\tz_t)=\|F(\tz_t)\|-L\eta\|F(\tz_{t-1})\|\\
        \geq&(1-2\eta L)\|F(\tz_t)\|.        
    \end{split}
    \]
\end{proof}
\begin{lemma}\label{lem:pegtraj}
    For the iterates of RPEG as in Equation \eqref{eq:rpeg}, with step-size $\eta\leq\frac{1}{32L}$, assume $d(\z_t,\z^*)\leq D$, Assumptions \ref{ass:rm} and \ref{ass:lips2} holds. Then we have
    \[
    \begin{split}
        &d(\z_{t+1},\z^*)\leq\frac{31D}{30}\\
        &d(\tz_t,\z^*)\leq\frac{16D}{15}\\
        &d(\tz_{t-1},\z^*)\leq\frac{6D}{5}.
    \end{split}
    \]
    We can also obtain $\z_t,\z_{t+1},\tz_t,\tz_{t-1}\in\D$ holds where $\D$ is defined in Definition \ref{def:def}.
\end{lemma}
\begin{proof}
    This lemma can be proved via a combination of the triangle inequality, Lemma \ref{lem:peggrad1}, Lemma \ref{lem:peggrad2} and Assumption \ref{ass:lips2}. For $d(\z_{t+1},\z^*)$:
    \[
    \begin{split}
        &d(\z_{t+1},\z^*)\leq d(\z_t,\z^*)+d(\z_{t+1},\z_t)\\
        \leq& D+\eta\|F(\tz_t)\|\leq D+\frac{\eta}{1-2L\eta}\|F(\z_t)\|\\
        \leq &D+\frac{\eta}{1-2L\eta}DL\leq D+\frac{1}{30}D=\frac{31D}{30}.
    \end{split}
    \]   
    We can bound $d(\tz_t,\z^*)$ in a similar way:
    \[
    \begin{split}
        &d(\tz_t,\z^*)\leq d(\z_t,\z^*)+d(\tz_t,\z_t)\\
        \leq& D+\eta\|F(\tz_{t-1})\|\\
        \leq& D+2\eta \|F(\tz_t)\|\\
        \leq&D+\frac{2}{30}D=\frac{16D}{15}.
    \end{split}
    \]
    The case of $d(\tz_{t-1},\z^*)$ is slightly more involved. First, by the triangle inequality,
    \[
    d(\tz_{t-1},\z^*)\leq d(\tz_{t-1},\z_{t-1})+d(\z_{t-1},\z^*).
    \]
    We bound both terms individually as:
    \[
    \begin{split}
        &d(\z_{t-1},\z^*)\leq d(\z_t,\z_{t-1})+d(\z_t,\z^*)\leq \eta\|F(\tz_{t-1})\|+D\leq\frac{16D}{15}\\
        &d(\tz_{t-1},\z_{t-1})=\eta\|F(\tz_{t-2})\|\leq 4\eta\|F(\tz_t)\|\leq\frac{4\eta}{1-2L\eta}DL\leq\frac{2D}{15}.
    \end{split}
    \]
    Thus,
    \[
    d(\tz_{t-1},\z^*)\leq d(\tz_{t-1},\z_{t-1})+d(\z_{t-1},\z^*)\leq\frac{2D}{15}+\frac{16D}{15}=\frac{6D}{5}.
    \]
\end{proof}

\subsection{Proof of Lemma \ref{lem:peg}}
\begin{proof}
We again prove the lemma by induction. The case of $t=0$ is obvious. Now, we assume $d(\z_t,\z^*)\leq D$ and we intend to show $d(\z_{t+1},\z^*)\leq D$. We define $c\coloneqq\frac{1}{6\sqrt{2}}$ for convenience.

By Lemma \ref{lem:cos2},
\begin{equation}
    d(\z_{t+1},\z^*)^2-d(\z_{t},\z^*)^2\leq 2\la\invexp_{\z_{t+1}}\z_{t},\invexp_{\z_{t+1}}\z^*\ra-\sigma(K,d(\z_t,\z_{t+1})+\min\lt\{d(\z_t,\z^*),d(\z_{t+1},\z^*)\rt\}) \cdot d(\z_{t},\z_{t+1})^2.
\end{equation}
Based on the monotonicity of $\sigma(K,\cdot)$ and $\eta\leq \frac{\bs}{32\bzt L}\leq\frac{1}{32L}$, we have
\begin{equation}\label{eq:pegsigma1}
    \begin{split}
        &-\sigma(K,d(\z_t,\z_{t+1})+\min\lt\{d(\z_t,\z^*),d(\z_{t+1},\z^*)\rt\})\\
        \leq&-\sigma(K,d(\z_t,\z^*)+d(\z_t,\z_{t+1}))\\
        \leq&-\sigma\lt(K,\frac{31D}{30}\rt)\leq-\bs,\\
    \end{split}
\end{equation}
where $d(\z_t,\z^*)+d(\z_t,\z_{t+1})\leq\frac{31D}{30}$ follows from the proof of Lemma \ref{lem:pegtraj}.
Thus, we can deduce
\begin{equation}\label{eq:pegbest1}
    d(\z_{t+1},\z^*)^2-d(\z_{t},\z^*)^2\leq 2\la\invexp_{\z_{t+1}}\z_{t},\invexp_{\z_{t+1}}\z^*\ra-\bs \cdot d(\z_{t},\z_{t+1})^2.
\end{equation}
We can decompose the term $2\la\invexp_{\z_{t+1}}\z_t,\invexp_{\z_{t+1}}\z^*\ra$ as:
\begin{equation}\label{eq:peg1}
    \begin{split}     &2\la\invexp_{\z_{t+1}}\z_t,\invexp_{\z_{t+1}}\z^*\ra=2\eta\la\G_{\z_t}^{\z_{t+1}}\G_{\tz_t}^{\z_t}F(\tz_t),\invexp_{\z_{t+1}}\z^*\ra\\       =&2\eta\la\G_{\z_t}^{\z_{t+1}}\G_{\tz_t}^{\z_t}F(\tz_t)-\G_{\tz_t}^{\z_{t+1}}F(\tz_t),\invexp_{\z_{t+1}}\z^*\ra+2\eta\la\G_{\tz_t}^{\z_{t+1}}F(\tz_t),\invexp_{\z_{t+1}}\z^*\ra\\ 
    \end{split}
\end{equation}
The first term on the RHS of Equation \eqref{eq:peg1} corresponds to the holonomy effect and can be bounded by Lemma \ref{lem:holo}. To this end, we first compute
\[
    \begin{split}
        &\min\{d(\z_t,\z_{t+1})+d(\z_{t+1},\tz_t),d(\z_t,\tz_t)+d(\z_{t+1},\tz_t)\}\\
        \leq&d(\z_t,\z_{t+1})+d(\z_{t+1},\tz_t)\leq 2d(\z_t,\z_{t+1})+d(\z_t,\tz_t)=2\eta\|F(\tz_t)\|+\eta\|F(\tz_{t-1})\|\leq 3\eta G\\
        \leq&3G\cdot \frac{1}{\sqrt{648K_mG^2}}=\frac{1}{6\sqrt{2}\cdot\sqrt{K_m}}=\frac{c}{\sqrt{K_m}}.
    \end{split}
    \]
    where $2\eta\|F(\tz_t)\|+\eta\|F(\tz_{t-1})\|\leq 3\eta G$ is due to Lemma \ref{lem:pegtraj} and Assumption \ref{ass:grad}.
    
Now we have
    \begin{equation}\label{eq:peg2}
        \begin{split}           &2\la\invexp_{\z_{t+1}}\z_t,\invexp_{\z_{t+1}}\z^*\ra\\
        \stackrel{(1)}{\leq}& 36c\sqrt{K_m}\cdot\|F(\tz_t)\|\cdot d(\tz_t,\z_{t+1})\cdot 2\eta d(\z^*,\z_{t+1})+2\eta\la\G_{\tz_t}^{\z_{t+1}}F(\tz_t),\invexp_{\z_{t+1}}\z^*\ra\\
        \stackrel{(2)}{\leq}&144cL\frac{31D}{30}\sqrt{K_m}\eta^3\|F(\tz_t)\|\cdot(2\|F(\tz_{t-1})\|+\|F(\tz_{t-2})\|)+2\eta\la\G_{\tz_t}^{\z_{t+1}}F(\tz_t),\invexp_{\z_{t+1}}\z^*\ra,\\
        \end{split}
    \end{equation}
    where the first inequality follows from Lemma \ref{lem:holo} and the second is a result of Lemma \ref{lem:peggrad3} and Lemma \ref{lem:pegtraj}.

    \[
    \begin{split}
        &\min\{d(\z_t,\z_{t+1})+d(\z_{t+1},\tz_t),d(\z_t,\tz_t)+d(\z_{t+1},\tz_t)\}\\
        \leq&d(\z_t,\z_{t+1})+d(\z_{t+1},\tz_t)\leq 2d(\z_t,\z_{t+1})+d(\z_t,\tz_t)=2\eta\|F(\tz_t)\|+\eta\|F(\tz_{t-1})\|\leq 3\eta G\\
        \leq&3G\cdot \frac{1}{\sqrt{648K_mG^2}}=\frac{1}{6\sqrt{2}\cdot\sqrt{K_m}}=\frac{c}{\sqrt{K_m}}.
    \end{split}
    \]
    while the second is a result of Lemma \ref{lem:peggrad3}.
    We  also achieve an upper bound on $2\eta\la\G_{\tz_t}^{\z_{t+1}}F(\tz_t),\invexp_{\z_{t+1}}\z^*\ra$ as follows:
    \begin{equation}\label{eq:peg3}
        \begin{split}            &2\eta\la\G_{\tz_t}^{\z_{t+1}}F(\tz_t),\invexp_{\z_{t+1}}\z^*\ra=2\eta\la F(\tz_t),\G_{\z_{t+1}}^{\tz_t}\invexp_{\z_{t+1}}\z^*\ra\\
        =&2\eta\la F(\tz_t),\G_{\z_{t+1}}^{\tz_t}\invexp_{\z_{t+1}}\z^*-\invexp_{\tz_t}\z^*\ra+2\eta\la F(\tz_t),\invexp_{\tz_t}\z^*\ra\\
        \leq &2\eta\bzt \|F(\tz_t)\|\cdot d(\tz_t,\z_{t+1})\\
        \leq&4\bzt L\eta^3\|F(\tz_t)\|\cdot (2\|F(\tz_{t-1})\|+\|F(\tz_{t-2})\|),
        \end{split}
    \end{equation}
    where the inequality is due to Lemma \ref{lem:hessian} and Assumption \ref{ass:mono2}, while the last equality is due to Lemma \ref{lem:peggrad3}. The correctness of the geometric distortion $\bzt$ can be verified in an analog way as Equation \eqref{eq:egzeta}. More specifically, 
    \[
    \begin{split}
        &d(\tz_t,\z_{t+1})+\min\lt\{d(\tz_t,\z^*),d(\z_{t+1},\z^*)\rt\}\\
        \leq& d(\tz_t,\z_{t+1})+\frac{31D}{30}\\
        \leq& d(\tz_t,\z_t)+d(\z_t,\z_{t+1})+\frac{31D}{30}\\
        =&\eta\|F(\tz_{t-1})\|+\eta\|F(\tz_t)\|+\frac{31D}{30}\\
        \leq& 3\eta\|F(\tz_t)\|+\frac{31D}{30}\\
        \leq& \frac{3\eta}{1-2L\eta}\|F(\z_t)\|\leq \frac{3\eta}{1-2\eta L}\cdot DL+D\leq \frac{11D}{10}\leq \frac{7D}{5}.
    \end{split}
    \]

    Combining Equations \eqref{eq:pegbest1}, \eqref{eq:peg2} and \eqref{eq:peg3}, we have
    \begin{equation}\label{eq:peg4}
        \begin{split} 
        &d(\z_{t+1},\z^*)^2-d(\z_{t},\z^*)^2\leq 2\la\invexp_{\z_{t+1}}\z_{t},\invexp_{\z_{t+1}}\z^*\ra-\bs \cdot d(\z_{t},\z_{t+1})^2\\
            \leq&(144cLD\sqrt{K_m}+4\bzt L)\eta^3\cdot \|F(\tz_t)\|\cdot(2\|F(\tz_{t-1})\|+\|F(\tz_{t-2})\|)-\bs\eta^2\|F(\tz_t)\|^2\\
            \leq& 8\lt(144cL\frac{31D}{30}\sqrt{K_m}+4\bzt L\rt)\eta^3\|F(\tz_t)\|^2-\bs\eta^2\|F(\tz_t)\|^2,
        \end{split}
    \end{equation}
    where the third inequality follows from Lemma \ref{lem:peggrad1}.  Remind that $c=\frac{1}{6\sqrt{2}}$, we can guarantee the RHS of Equation \eqref{eq:peg4} to be non-positive by choosing
    \[
    \eta\leq\frac{\bs}{141LD\sqrt{K_m}+32\bzt L}.
    \]
    Since
    \[
    \eta\leq\min\lt\{\frac{\bs}{141LD\sqrt{K_m}+32\bzt L}, \frac{1}{\sqrt{648K_mG^2}}\rt\},
    \]
    already satisfies this requirement, by induction, we know $d(\z_t,\z^*)\leq D$ holds for $t\geq 0$.   
\end{proof}
\subsection{Proof of Lemma \ref{lem:pegnorm}}
\begin{proof}
First, we note that the step-size $\eta$ already satisfies the requirement of Lemma \ref{lem:egbest}, so we know $\z_t,\z_{t+1},\tz_t,\tz_{t-1}\in\D$ holds for any $t\geq 0$ by combining Lemmas \ref{lem:egtraj} and \ref{lem:egbest}. This is important, because Assumptions \ref{ass:mono2} and \ref{ass:grad} only hold on $\D$.
    By Assumption \ref{ass:mono2}, we have
    \[
    \begin{split}
        0\leq&\la\G_{\z_{t+1}}^{\z_{t}}F(\z_{t+1})-F(\z_{t}),\invexp_{\z_{t}}\z_{t+1}\ra\\
        =&-\eta\la\G_{\z_{t+1}}^{\z_{t}}F(\z_{t+1})-F(\z_{t}),\G_{\tz_{t}}^{\z_{t}}F(\tz_{t})\ra,
    \end{split}
    \]
    which is an analog of the first inequality in Equation \eqref{eq:pegeq0}.
    To show the second inequality, by Assumption \ref{ass:lips2},
    \[
    \|\G_{\z_{t+1}}^{\tz_{t}}F(\z_{t+1})-F(\tz_{t})\|^2\leq L^2 d(\z_{t+1},\tz_{t})^2,
    \]
    and our goal is 
    \[
    \|\G_{\z_{t+1}}^{\z_{t}}F(\z_{t+1})-\G_{\tz_{t}}^{\z_{t}}F(\tz_{t})\|^2\leq O(1)\cdot d(\z_{t+1},\tz_{t})^2.
    \]
    First, we have
    \begin{equation}\label{eq:normpeg1}
        \begin{split}
            &\|\G_{\tz_{t}}^{\z_{t}}F(\tz_{t})-\G_{\z_{t+1}}^{\z_{t}}F(\z_{t+1})\|^2\\
        =&\|\G_{\tz_{t}}^{\z_{t}}F(\tz_{t})-\G_{\tz_{t}}^{\z_{t}}\G_{\z_{t+1}}^{\tz_{t}}F(\z_{t+1})+\G_{\tz_{t}}^{\z_{t}}\G_{\z_{t+1}}^{\tz_{t}}F(\z_{t+1})-\G_{\z_{t+1}}^{\z_{t}}F(\z_{t+1})\|^2\\
        {\leq} &2\|\G_{\tz_{t}}^{\z_{t}}F(\tz_{t})-\G_{\tz_{t}}^{\z_{t}}\G_{\z_{t+1}}^{\tz_{t}}F(\z_{t+1})\|^2+2\|\G_{\z_{t}}^{\z_{t+1}}\G_{\tz_{t}}^{\z_{t}}\G_{\z_{t+1}}^{\tz_{t}}F(\z_{t+1})-F(\z_{t+1})\|^2\\
        \end{split}
    \end{equation}
    where the inequality follows from $\|\a+\b\|^2\leq 2(\|\a\|^2+\|\b\|^2)$.
Since
\[
\begin{split}
   &\min\{d(\z_t,\z_{t+1})+d(\z_{t+1},\tz_t),d(\z_t,\tz_t)+d(\z_{t+1},\tz_t)\}\\
   \leq& d(\z_t,\z_{t+1})+d(\z_{t+1},\tz_t)\leq 2d(\z_t,\z_{t+1})+d(\z_t,\tz_t)\leq 3\eta G\\
   \leq &3G\cdot\frac{1}{\sqrt{648K_mG^2}}=\frac{1}{6\sqrt{2}\cdot\sqrt{K_m}}\\
   \coloneqq& \frac{c}{\sqrt{K_m}}
\end{split}
\]
By Lemma \ref{lem:holo},
    \begin{equation}\label{eq:normpeg2}
    \begin{split}    &\|\G_{\z_{t}}^{\z_{t+1}}\G_{\tz_{t}}^{\z_{t}}\G_{\z_{t+1}}^{\tz_{t}}F(\z_{t+1})-F(\z_{t+1})\|\\
\leq& 36K_m\|F(\z_{t+1})\|\cdot \min\{d(\z_t,\z_{t+1})+d(\z_{t+1},\tz_t),d(\z_t,\tz_t)+d(\z_{t+1},\tz_t)\}\cdot d(\z_{t+1},\tz_t)\\
\leq & 36c\sqrt{K_m}G\cdot d(\tz_t,\z_{t+1})
\end{split}
\end{equation}
Combining Equations \eqref{eq:normpeg1} and \eqref{eq:normpeg2} yields
\begin{equation}\label{eq:normpeg3}
    \begin{split}
        &\|\G_{\tz_{t}}^{\z_{t}}F(\tz_{t})-\G_{\z_{t+1}}^{\z_{t}}F(\z_{t+1})\|^2\\
        \stackrel{(1)}{\leq} &2\|\G_{\tz_{t}}^{\z_{t}}F(\tz_{t})-\G_{\tz_{t}}^{\z_{t}}\G_{\z_{t+1}}^{\tz_{t}}F(\z_{t+1})\|^2+2\|\G_{\z_{t}}^{\z_{t+1}}\G_{\tz_{t}}^{\z_{t}}\G_{\z_{t+1}}^{\tz_{t}}F(\z_{t+1})-F(\z_{t+1})\|^2\\
        \stackrel{(2)}{\leq}&(2L^2+2592c^2 K_mG^2)d(\tz_{t},\z_{t+1})^2\\
        \stackrel{(3)}{\leq}&(2L^2+2592c^2 K_mG^2)\cdot 4\|\invexp_{\z_{t}}\tz_{t}-\invexp_{\z_{t}}\z_{t+1}\|^2\\
        =&(8L^2+10368c^2 K_mG^2)\eta^2\|\G_{\tz_t}^{\z_t}F(\tz_{t})-\G_{\tz_{t-1}}^{\z_{t}}F(\tz_{t-1})\|^2,
    \end{split}
\end{equation}
where the first inequality is due to $\|\a+\b\|^2\leq 2(\|\a\|^2+\|\b\|^2)$, the second is by Assumption \ref{ass:lips2} and Lemma \ref{lem:holo}, the third is due to Lemma \ref{lem:trig}. Thus, we have
\begin{equation}\label{eq:pegeq1}
    0\leq \la F(\z_{t})-\G_{\z_{t+1}}^{\z_{t}}F(\z_{t+1}),\G_{\tz_{t}}^{\z_{t}}F(\tz_{t})\ra
\end{equation}
and
\begin{equation}\label{eq:pegeq2}
    \|\G_{\tz_{t}}^{\z_{t}}F(\tz_{t})-\G_{\z_{t+1}}^{\z_{t}}F(\z_{t+1})\|^2\leq (8L^2+10368 c^2K_mG^2)\eta^2 \|\G_{\tz_{t-1}}^{\z_{t}} F(\tz_{t-1})-\G_{\tz_{t}}^{\z_{t}}F(\tz_{t})\|^2.
\end{equation}
Adding two times of Equation \eqref{eq:pegeq1} and three times of Equation \eqref{eq:pegeq2} together yields
\[
\begin{split}
    &3\|\G_{\tz_{t}}^{\z_{t}}F(\tz_{t})-\G_{\z_{t+1}}^{\z_{t}}F(\z_{t+1})\|^2\\
    \leq& 2\la F(\z_{t})-\G_{\z_{t+1}}^{\z_{t}}F(\z_{t+1}),\G_{\tz_{t}}^{\z_{t}}F(\tz_{t})\ra+3(8L^2+10368c^2 K_mG^2)\eta^2 \|\G_{\tz_{t-1}}^{\z_{t}} F(\tz_{t-1})-\G_{\tz_{t}}^{\z_{t}}F(\tz_{t})\|^2\\
    =&\|F(\z_{t})\|^2-\|F(\z_{t})-\G_{\tz_{t}}^{\z_{t}}F(\tz_{t})\|^2+\|\G_{\z_{t+1}}^{\z_{t}}F(\z_{t+1})-\G_{\tz_{t}}^{\z_{t}}F(\tz_{t})\|^2-\|F(\z_{t+1})\|^2\\
    &+3(8L^2+10368c^2 K_mG^2)\eta^2 \|\G_{\tz_{t-1}}^{\z_{t}} F(\tz_{t-1})-\G_{\tz_{t}}^{\z_{t}}F(\tz_{t})\|^2\\
\end{split}
\]
where we use $2\la \a,\b\ra=\|\a\|^2+\|\b\|^2-\|\a-\b\|^2$ holds for $\a,\b$ in the same tangent space. Rearranging, we obtain
\begin{equation}\label{eq:decreasenorm}
   \begin{split}
        &\|F(\z_{t+1})\|^2+2\|\G_{\tz_{t}}^{\z_{t}}F(\tz_{t})-\G_{\z_{t+1}}^{\z_{t}}F(\z_{t+1})\|^2\\
        \leq &\|F(\z_{t})\|^2-\|F(\z_{t})-\G_{\tz_{t}}^{\z_{t}}F(\tz_{t})\|^2+3(8L^2+10368c^2 K_mG^2)\eta^2\|\G_{\tz_{t-1}}^{\z_{t}} F(\tz_{t-1})-\G_{\tz_{t}}^{\z_{t}}F(\tz_{t})\|^2.
   \end{split}
\end{equation}
Applying $-\|\a-\b\|^2\leq-\frac{1}{1+\alpha}\|\a\|^2+\frac{1}{\alpha}\|\b\|^2$ with $\a=\G_{\tz_{t}}^{\z_{t}}F(\tz_{t})-\G_{\tz_{t-1}}^{\z_{t}}F(\tz_{t-1})$, $\b=F(\z_{t})-\G_{\tz_{t-1}}^{\z_{t}}F(\tz_{t-1})$ and $\alpha=\frac{1}{2}$, we have
\[
-\|F(\z_{t})-\G_{\tz_{t}}^{\z_{t}}F(\tz_{t})\|^2\leq-\frac{2}{3}\|\G_{\tz_{t-1}}^{\z_{t}}F(\tz_{t-1})-\G_{\tz_{t}}^{\z_{t}}F(\tz_{t})\|^2+2\|F(\z_{t})-\G_{\tz_{t-1}}^{\z_{t}}F(\tz_{t-1})\|^2.
\]
Plugging into Equation \eqref{eq:decreasenorm}, rearranging, we have
\begin{equation}\label{eq:pepauxi1}
    \begin{split}        \|F(\z_{t+1})\|^2+2\|\G_{\tz_{t}}^{\z_{t}}F(\tz_{t})-\G_{\z_{t+1}}^{\z_{t}}F(\z_{t+1})\|^2\leq &\|F(\z_{t})\|^2+2\|F(\z_{t})-\G_{\tz_{t-1}}^{\z_{t}}F(\tz_{t-1})\|^2\\
        &+3\lt(8L^2+10368c^2 K_mG^2\eta^2-\frac{2}{9}\rt)\|\G_{\tz_{t}}^{\z_{t}}F(\tz_{t})-\G_{\tz_{t-1}}^{\z_{t}}F(\tz_{t-1})\|^2.
    \end{split}
\end{equation}
We are nearing the completion of the proof, but a subtle issue arises. The left-hand side (LHS) of Equation \eqref{eq:pepauxi1} should feature
$2\|F(\z_{t+1})-\G_{\tz_{t}}^{\z_{t+1}}F(\tz_{t})\|^2$ as opposed to $2\|\G_{\tz_{t}}^{\z_{t}}F(\tz_{t})-\G_{\z_{t+1}}^{\z_{t}}F(\z_{t+1})\|^2$ to better suit the subsequent Lyapunov analysis. Fortunately, this discrepancy can be rectified by taking a closer look at the holonomy effect and bounding it appropriately. To that end, we present the following calculations:

\begin{equation}\label{eq:pepauxi2}
    \begin{split}
        &\|\G_{{\tz}_{t}}^{\z_{t+1}}F(\tz_{t})-F(\z_{t+1})\|^2=\|\G_{\tz_{t}}^{\z_{t+1}}F(\tz_{t})-\G_{\tz_{t}}^{\z_{t+1}}\G_{\z_{t}}^{\tz_{t}}\G_{\z_{t+1}}^{\z_{t}}F(\z_{t+1})\|^2\\
        &+\la 2\G_{\tz_{t}}^{\z_{t+1}}F(\tz_{t})-F(\z_{t+1})-\G_{\tz_{t}}^{\z_{t+1}}\G_{\z_{t}}^{\tz_{t}}\G_{\z_{t+1}}^{\z_{t}}F(\z_{t+1}),\G_{\tz_{t}}^{\z_{t+1}}\G_{\z_{t}}^{\tz_{t}}\G_{\z_{t+1}}^{\z_{t}}F(\z_{t+1})-F(\z_{t+1})\ra\\
        \leq& \|\G_{\tz_{t}}^{\z_{t+1}}F(\tz_{t})-\G_{\tz_{t}}^{\z_{t+1}}\G_{\z_{t}}^{\tz_{t}}\G_{\z_{t+1}}^{\z_{t}}F(\z_{t+1})\|^2+2\la\G_{\tz_{t}}^{\z_{t+1}}F(\tz_{t})-F(\z_{t+1}),\G_{\tz_{t}}^{\z_{t+1}}\G_{\z_{t}}^{\tz_{t}}\G_{\z_{t+1}}^{\z_{t}}F(\z_{t+1})-F(\z_{t+1})\ra\\
        =& \|\G^{\z_{t}}_{\tz_{t}}F(\tz_{t})-\G_{\z_{t+1}}^{\z_{t}}F(\z_{t+1})\|^2+2\la\G_{\tz_{t}}^{\z_{t+1}}F(\tz_{t})-F(\z_{t+1}),\G_{\tz_{t}}^{\z_{t+1}}\G_{\z_{t}}^{\tz_{t}}\G_{\z_{t+1}}^{\z_{t}}F(\z_{t+1})-F(\z_{t+1})\ra\\
        \leq&\|\G^{\z_{t}}_{\tz_{t}}F(\tz_{t})-\G_{\z_{t+1}}^{\z_{t}}F(\z_{t+1})\|^2+2L\cdot d(\tz_{t},\z_{t+1})\cdot \|F(\z_{t+1})-\G_{\tz_{t}}^{\z_{t+1}}\G_{\z_{t}}^{\tz_{t}}\G_{\z_{t+1}}^{\z_{t}}F(\z_{t+1})\|\\
        \leq &\|\G^{\z_{t}}_{\tz_{t}}F(\tz_{t})-\G_{\z_{t+1}}^{\z_{t}}F(\z_{t+1})\|^2+2L\cdot 36c\sqrt{K_m}G d(\tz_{t},\z_{t+1})^2\\
        \leq &\|\G^{\z_{t}}_{\tz_{t}}F(\tz_{t})-\G_{\z_{t+1}}^{\z_{t}}F(\z_{t+1})\|^2+2L\cdot 36c\sqrt{K_m}G\cdot 4\|\invexp_{\z_{t}}\tz_{t}-\invexp_{\z_{t}}\z_{t+1}\|^2\\
        =&\|\G^{\z_{t}}_{\tz_{t}}F(\tz_{t})-\G_{\z_{t+1}}^{\z_{t}}F(\z_{t+1})\|^2+288cGL\sqrt{K_m} \eta^2\|\G_{\tz_{t-1}}^{\z_{t}}F(\tz_{t-1})-\G_{\tz_{t}}^{\z_{t}}F(\tz_{t})\|^2
    \end{split}
\end{equation}
where the first equality and the first inequality follows from
\[
\|\a-\b\|^2=\|\a-\c\|^2+\la 2\a-\b-\c,\c-\b\ra\leq\|\a-\c\|^2+2\la\a-\b,\c-\b\ra
\]
holds for $\a,\b,\c\in T_{\z_{t+1}}\M$.
Combining Equations \eqref{eq:pepauxi1} and \eqref{eq:pepauxi2}, we get
\[
\begin{split}
    &\|F(\z_{t+1})\|^2+2\|F(\z_{t+1})-\G_{\tz_{t}}^{\z_{t+1}}F(\tz_{t})\|^2\\
    \leq&\|F(\z_{t})\|^2+2\|F(\z_{t})-\G_{\tz_{t-1}}^{\z_{t}}F(\tz_{t-1})\|^2\\
    &+\lt(\lt(24L^2+31104c^2K_mG^2+576cGL\sqrt{K_m}\rt)\eta^2-\frac{2}{3}\rt)\|\G_{\tz_{t}}^{\z_{t}}F(\tz_{t})-\G_{\tz_{t-1}}^{\z_{t}}F(\tz_{t-1})\|^2\\
    =&\|F(\z_{t})\|^2+2\|F(\z_{t})-\G_{\tz_{t-1}}^{\z_{t}}F(\tz_{t-1})\|^2\\
    &+\lt(\lt(24L^2+432K_mG^2+48GL\sqrt{2K_m}\rt)\eta^2-\frac{2}{3}\rt)\|\G_{\tz_{t}}^{\z_{t}}F(\tz_{t})-\G_{\tz_{t-1}}^{\z_{t}}F(\tz_{t-1})\|^2\\
\end{split}
\]
where for the equality, we plug in $c=\frac{1}{6\sqrt{2}}$.
\end{proof}
\subsection{Proof of Lemma \ref{lem:bestpeg}}
\begin{proof}
Similar to the proof of Lemma \ref{lem:peg}, we define $c\coloneqq\frac{1}{6\sqrt{2}}$. Since $\eta\leq \frac{1}{\sqrt{36L^2+648K_mG^2+72\sqrt{2K_m}GL}}\leq\frac{1}{\sqrt{648K_mG^2}}$, by Lemma \ref{lem:pegnorm}, Lemma \ref{lem:cos2},  the definition of $\Phi_t$, and an analog of Equation \eqref{eq:pegsigma1}, we have
    \begin{equation}\label{eq:iterpeg1}
        \begin{split}
             \Phi_{t+1}-\Phi_t=&d(\z_{t+1},\z^*)^2+\lambda (t+1)\eta^2\lt(\|F(\z_{t+1})\|^2+2\|F(\z_{t+1})-\G_{\tz_{t}}^{\z_{t+1}}F(\tz_{t})\|^2\rt)\\
    &-\lt(d(\z_{t},\z^*)^2+\lambda t\eta^2\lt(\|F(\z_{t})\|^2+2\|F(\z_{t})-\G_{\tz_{t-1}}^{\z_{t}}F(\tz_{t-1})\|^2\rt)\rt)\\
    \leq &2\la\invexp_{\z_{t+1}}\z_t,\invexp_{\z_{t+1}}\z^*\ra-\bs\eta^2\|F(\tz_t)\|^2\\
    &+\lambda\eta^2(\|F(\z_t)\|^2+2\|F(\z_t)-\G_{\tz_{t-1}}^{\z_t}F(\tz_{t-1})\|^2)\\   &+\lambda(t+1)\lt(\lt(24L^2+432K_mG^2+48GL\sqrt{2K_m}\rt)\eta^2-\frac{2}{3}\rt)\|\G_{\tz_{t}}^{\z_{t}}F(\tz_{t})-\G_{\tz_{t-1}}^{\z_{t}}F(\tz_{t-1})\|^2.
        \end{split}
    \end{equation}

    By $\eta\leq \frac{1}{\sqrt{36L^2+648K_mG^2+72\sqrt{2K_m}GL}}$, we have
    \[
    \lt(\lt(24L^2+432K_mG^2+48GL\sqrt{2K_m}\rt)\eta^2-\frac{2}{3}\rt)\leq 0,
    \]
    thus, the last term on the RHS of Equation \eqref{eq:iterpeg1} vanishes. Note that since $\eta$ satisfies the requirement in Lemma \ref{lem:peg}, by Equation \eqref{eq:peg4}, we have
    \begin{equation}\label{eq:peg5}
        \begin{split} 
        &d(\z_{t+1},\z^*)^2-d(\z_{t},\z^*)^2\leq 2\la\invexp_{\z_{t+1}}\z_{t},\invexp_{\z_{t+1}}\z^*\ra-\bs \cdot d(\z_{t},\z_{t+1})^2\\
            \leq&(144cLD\sqrt{K_m}+4\bzt L)\eta^3\cdot \|F(\tz_t)\|\cdot(2\|F(\tz_{t-1})\|+\|F(\tz_{t-2})\|)-\bs\eta^2\|F(\tz_t)\|^2\\
            \leq& 8\lt(144cL\frac{31D}{30}\sqrt{K_m}+4\bzt L\rt)\eta^3\|F(\tz_t)\|^2-\bs\eta^2\|F(\tz_t)\|^2,
        \end{split}
    \end{equation}

    By $\eta\leq\frac{1}{32L}$ and Lemmas \ref{lem:peggrad1} and \ref{lem:peggrad2}, we also have
    \[
    \|F(\z_t)\|\leq(1+2L\eta)\|F(\tz_t)\|\leq\frac{17}{16}\|F(\tz_t)\|
    \]
    and 
    \[
    \begin{split}
        \|F(\z_t)-\G_{\tz_{t-1}}^{\z_t}F(\tz_{t-1})\|\leq& L d(\z_t,\tz_{t-1})\leq L(d(\z_t,\z_{t-1})+d(\z_{t-1},\tz_{t-1})\\
        =& L\eta(\|F(\tz_{t-1})\|+\|F(\tz_{t-2})\|)\leq 6L\eta\|F(\tz_t)\|.
    \end{split}
    \]
    Thus,
    \begin{equation}\label{eq:bestpeg4}
        \begin{split}
            &\lambda\eta^2\|F(\z_t)\|^2+2\lambda\eta^2\|F(\z_t)-\G_{\tz_{t-1}}^{\z_t}F(\tz_{t-1})\|^2\\
            \leq&\frac{289}{256}\lambda\eta^2\|F(\tz_t)\|^2+2\lambda\eta^2\cdot 36L^2\eta^2\|F(\tz_t)\|^2\\
            =&\lt(\frac{289}{256}\lambda\eta^2+72\lambda L^2\eta^4\rt)\|F(\tz_t)\|^2.
        \end{split}
    \end{equation}
    Combining Equations \eqref{eq:iterpeg1}, \eqref{eq:peg5} and \eqref{eq:bestpeg4} and choosing $\lambda=\frac{\bs}{16}$, we have
    \[
    \begin{split}
        \Phi_{t+1}-\Phi_t\leq& 8\lt(144cL\frac{31D}{30}\sqrt{K_m}+4\bzt L\rt)\eta^3\|F(\tz_t)\|^2-\bs\eta^2\|F(\tz_t)\|^2+\lt(\frac{289}{256}\lambda\eta^2+72\lambda L^2\eta^4\rt)\|F(\tz_t)\|^2\\
        \leq& 8\lt(144cL\frac{31D}{30}\sqrt{K_m}+4\bzt L\rt)\eta^3\|F(\tz_t)\|^2-\bs\eta^2\|F(\tz_t)\|^2+\lt(\frac{\bs}{14}\eta^2+\frac{9\bs}{64}L\eta^3\rt)\|F(\tz_t)\|^2\\
        \leq &\lt(\lt(141LD\sqrt{K_m}+32\bzt L+\frac{9\bs L}{64}\rt)\eta^3-\frac{13\bs\eta^2}{14}\rt)\cdot\|F(\tz_t)\|^2\\
    \end{split}
    \]
    where we recall $c=\frac{1}{6\sqrt{2}}$. Now, we find by taking
    \[
    \eta=\frac{\bs}{152LD\sqrt{K_m}+35\bzt L},
    \]
    $\Phi_{t+1}\leq\Phi_t$ holds. This requirement is always satisfied because
    \[
    \eta\leq\min\lt\{\frac{\bs}{152LD\sqrt{K_m}+35\bzt L},\frac{1}{\sqrt{36L^2+648K_mG^2+72\sqrt{2K_m}GL}}\rt\}\leq \frac{\bs}{152LD\sqrt{K_m}+35\bzt L}
    \]
    always holds.
\end{proof}
\subsection{Proof of Theorem \ref{thm:peglast}}
\begin{proof}
By Lemma \ref{lem:bestpeg}, we have
\[
\begin{split}
    &d(\z_T,\z^*)^2+\frac{\bs}{16}T\eta^2\lt(\|F(\z_T)\|^2+2\|F(\z_T)-\G_{\tz_{T-1}}^{\z_T}F(\tz_{T-1})\|^2\rt)\\
    =&\Phi_T\leq\Phi_{T-1}\leq\dots\leq \Phi_0=d(\z_0,\z^*)^2\leq D^2.
\end{split}
\]
Thus,
\[
\|F(\z_T)\|^2\leq D^2\cdot\frac{16}{\bs T\eta^2}=\mhl{O\lt(\frac{\bzt^2}{\bs^3T}\rt)}
\]
and $\|F(\z_T)\|=\mhl{O\lt(\frac{\bzt}{\sqrt{\bs^3T}}\rt)}$.

\end{proof}
\subsection{Proof of Theorem \ref{thm:rpeg}}
Similar to REG, we need the following auxillary lemma to establish the average-iterate convergence of RPEG.
\begin{lemma}\label{lem:pegprimal}
    Under Assumptions \ref{ass:rm},  \ref{ass:lips2} and \ref{ass:grad}. For the iterates of RPEG as in Equation \eqref{eq:rpeg} with 
    \[
    \eta\leq\min\lt\{\frac{\bs}{192LD\sqrt{2K_m}+32\bzt L}, \frac{1}{\sqrt{648K_mG^2}}\rt\},
    \]
    we have
    \[
    d(\z_{t+1},\z)^2-d(\z_t,\z)^2\leq 2\eta\la F(\tz_t),\invexp_{\tz_t}\z\ra.
    \]
    holds for any $t\geq 0$ and $\z\in\B(\z^*,D)$ where $\B(\z^*,D)$ denotes the geodesic ball with center $\z^*$ and radius $D$.
\end{lemma}
\begin{proof}
The proof is similar to that of Lemma \ref{lem:peg}.
    Combining Equations \eqref{eq:pegbest1}, \eqref{eq:peg2} and \eqref{eq:peg3}, and replacing $\z^*$ with $\z$, we have
    \begin{equation}\label{eq:pegprim}
        \begin{split}
            &d(\z_{t+1},\z)^2-d(\z_{t},\z)^2\leq 2\la\invexp_{\z_{t+1}}\z_{t},\invexp_{\z_{t+1}}\z\ra-\bs d(\z_{t},\z_{t+1})^2\\
            \le&8\lt(\frac{144}{6\sqrt{2}} L\cdot d(\z_{t+1},\z)\sqrt{K_m}+4\bzt L\rt)\eta^3\|F(\tz_t)\|^2-\bs\eta^2\|F(\tz_t)\|^2+2\eta\la F(\tz_t),\invexp_{\tz_t}\z\ra,\\
        =&\lt(96\sqrt{2}L\cdot d(\z_{t+1},\z)\sqrt{K_m}+4\bzt L\rt)\eta^3\|F(\tz_t)\|^2-\bs\eta^2\|F(\tz_t)\|^2+2\eta\la F(\tz_t),\invexp_{\tz_t}\z\ra,\\
        \leq &\lt(192\sqrt{2}L\cdot D\sqrt{K_m}+4\bzt L\rt)\eta^3\|F(\tz_t)\|^2-\bs\eta^2\|F(\tz_t)\|^2+2\eta\la F(\tz_t),\invexp_{\tz_t}\z\ra,\\
        \end{split}
    \end{equation}
    where the last inequality is due to Lemma \ref{lem:peg} and $d(\z_{t+1},\z)\leq d(\z_{t+1},\z^*)+d(\z^*,\z)\leq 2D$. It is straightforward to see, for any
    \[
    \eta\leq \frac{\bs}{192LD\sqrt{2K_m}+32\bzt L},
    \]
    we have
    \[
    d(\z_{t+1},\z)^2-d(\z_t,\z)^2\leq 2\eta\la F(\tz_t),\invexp_{\tz_t}\z\ra,
    \]
        which concludes the proof.
\end{proof}
Now, we can start to prove Theorem \ref{thm:rpeg}.
\begin{proof}
The proof closely parallels that of Theorem \ref{thm:reg}, but we provide details for the sake of completeness.
Denote $\z=\big(\begin{smallmatrix}
    \x\\ \y\end{smallmatrix}\big)$ and $\Z=\B(\z^*,D)$. We start with the last-iterate convergence, 
\[
\begin{split}
    &\max_{\y\in\Y}f(\x_T,\y)-\min_{\x\in\X}f(\x,\y_T)\leq\max_{\y\in\Y}\la\nabla_{\y}f(\x_T,\y_T),\invexp_{\y_T}\y\ra+\max_{\x\in\X}\la\nabla_{\x}f(\x_T,\y_T),-\invexp_{\x_T}\x\ra\\
    =&\max_{\z\in\Z}\la -F(\z_T),\invexp_{\z_T}\z\ra
    \leq \|F(\z_T)\|\cdot\max_{\z\in\Z}d(\z_T,\z)\leq \|F(\z_T)\|\cdot\lt(d(\z_T,\z^*)+\max_{\z\in\Z}d(\z^*,\z)\rt)\\
    \leq& \|F(\z_T)\|\cdot 2D=\mhl{O\lt(\frac{\bzt}{\sqrt{\bs^3T}}\rt)}.
\end{split}
\]
where the last equality is due to Theorem \ref{thm:peglast}. For the average-iterate convergence, 
\[
\begin{split}
    f(\bx_T,\y)-f(\x,\by_T)\stackrel{(1)}{\leq}&\frac{1}{T}\lt(\sumt f(\tx_t,\y)-\sumt f(\x,\ty_t)\rt)\\
    \stackrel{(2)}{\leq}&\frac{1}{T}\sumt\la-F(\tz_t),\invexp_{\tz_t}\z\ra\\
    \stackrel{(3)}{\leq}&\frac{1}{2\eta T}\sumt\lt(d(\z_t,\z)^2-d(\z_{t+1},\z)^2\rt)\\
    \stackrel{(4)}{\leq}&\frac{d(\z_0,\z)^2}{2\eta T}\leq \frac{(2D)^2}{2\eta T}=\frac{2D^2}{\eta T}=\mhl{O\lt(\frac{\bzt}{\bs T}\rt)},
\end{split}
\]
where the first inequality comes from a nested application of Jensen's inequality for gsc-convex functions:
\[
f(\bx_t,\y)\leq\frac{1}{t}f(\tx_t,\y)+\frac{t-1}{t}f(\bx_{t-1},\y),
\] the second is by the gsc-convexity, and the third comes from Lemma \ref{lem:pegprimal}. We also note that
\[
    \eta\leq\min\lt\{\frac{\bs}{192LD\sqrt{2K_m}+35\bzt L},\frac{1}{\sqrt{36L^2+648K_mG^2+72\sqrt{2K_m}GL}}\rt\}
    \]
    satisfies the requirement for $\eta$ as specified in Lemma \ref{lem:pegprimal}.
\end{proof}
\section{Technical Lemmas}
\subsection{Best-iterate Convergence of RCEG}
\label{app:best}
For completeness, we provide the $O\left(\frac{1}{\sqrt{T}}\right)$ rate for the best-iterate convergence of RCEG as follows. The proof is inspired by Proposition 5 of \citet{martinez2023minimax}.

\begin{thm}
    Consider a Riemannian manifold $\M$ with sectional curvature in $[\kappa,K]$, $D=d(\z_0,\z^*)$. If $K>0$, we require that $D<\frac{2\pi}{2\sqrt{K}}$. Let $\bzt=\zeta(\k,\frac{3D}{2})$ and $\bs=\sigma(K,\frac{3D}{2})$ be geometric constants defined in Lemma \ref{lem:hessian} and Lemma \ref{lem:cos2}. With $\eta\leq \sqrt{\frac{\bs}{4\bzt L^2}}$, RCEG defined by
    \[
   \begin{split}
        &\tz_t=\expmap_{\z_t}(-\eta F(\z_t))\\
        &\z_{t+1}=\expmap_{\tz_t}\lt(-\eta F(\tz_t)+\invexp_{\tz_t}\z_t\rt).
   \end{split}
    \]
    
    achieves $O\lt(\frac{1}{\sqrt{T}}\rt)$ best-iterate convergence for Riemannian variational inequality problems.
    
\end{thm}
\begin{proof}
    We use mathematical induction to establish $d(\z_t,\z^*)\leq D$ holds for any $t\geq 0$. The base case $t=0$ is straightforward. Assuming that $d(\z_t,\z^*) \leq D$ holds, we proceed to show that
    
    \begin{equation}\label{eq:traj}
            \begin{split}
        d(\tz_t,\z^*)&\leq d(\tz_t,\z_t)+d(\z_t,\z^*)\\
        &=\eta\|F(\z_t)\|+d(\z_t,\z^*)\\
        &=\eta\|F(\z_t)-\G_{\z^*}^{\z_t}F(\z^*)\|+d(\z_t,\z^*)\\
        &\leq \eta L d(\z_t,\z^*)+d(\z_t,\z^*)\\
        &=(1+\eta L)d(\z_t,\z^*)\\
        &\leq\frac{3}{2}d(\z_t,\z^*),
    \end{split}
    \end{equation}
    where the second inequality is due to the L-Lipschitzness of $F$ and the third inequality follows from $\eta\leq \sqrt{\frac{\bs}{4\bzt L^2}}\leq\frac{1}{2L}$. 
    
    Now by Equation \eqref{eq:traj}, Lemma \ref{lem:cos1} and Lemma \ref{lem:cos2},
    \begin{equation}\label{eq:best1}
            \begin{split}
        &2\la\invexp_{\tz_t}\z^*,\invexp_{\tz_t}\z_{t+1}-\invexp_{\tz_t}\z_t\ra\\
        \leq &\bzt d(\tz_t,\z_{t+1})^2-\bs d(\tz_t,\z_t)^2+d(\z^*,\z_t)^2-d(\z^*,\z_{t+1})^2\\
        =&\bzt\eta^2\|F(\tz_t)-\G_{\z_t}^{\tz_t}F(\z_t)\|^2-\bs d(\tz_t,\z_t)^2+d(\z^*,\z_t)^2-d(\z^*,\z_{t+1})^2\\
        \leq &\lt(\bzt\eta^2L^2-\bs\rt)d(\tz_t,\z_t)^2+d(\z^*,\z_t)^2-d(\z^*,\z_{t+1})^2.\\
    \end{split}
    \end{equation}
On the other hand, we have
\begin{equation}\label{eq:best2}
    \begin{split}
        &\la\invexp_{\tz_t}\z^*,\invexp_{\tz_t}\z_{t+1}-\invexp_{\tz_t}\z_t\ra\\
        =&\la\invexp_{\tz_t}\z^*,-\eta F(\tz_t)\ra\\
        =&\eta\la\G_{\z^*}^{\tz_t}F(\z^*)-F(\tz_t),\invexp_{\tz_t}\z^*\ra \geq 0.
    \end{split}
\end{equation}
Combining Equation \eqref{eq:best1} and Equation \eqref{eq:best2}, we have
\begin{equation}\label{eq:best3}
    d(\z_t,\z^*)^2\geq d(\z_{t+1},\z^*)^2+(\bs-\bzt\eta^2L^2)d(\tz_t,\z_t)^2.
\end{equation}

Given that $\bzt\eta^2L^2 \leq \bs$, it follows that $d(\z_{t+1},\z^*) \leq d(\z_t,\z^*) \leq D$, which completes the induction step. Summing over Equation \eqref{eq:best3}, we obtain
\[
\begin{split}
    d(\z_0,\z^*)^2\geq&(\bs-\bzt\eta^2L^2)\sum_{t=0}^{T-1}d(\tz_t,\z_t)^2\\
    =&\eta^2(\bs-\bzt\eta^2L^2)\sum_{t=0}^{T-1}\|F(\z_t)\|^2\\
    \geq&\eta^2(\bs-\bzt\eta^2L^2)T\cdot \min_{t'\in[T]}\|F(\z_t')\|^2,
\end{split}
\]
where the final inequality demonstrates that the best-iterate convergence rate of RCEG is $O\left(\frac{\sqrt{\bzt}}{\bs\sqrt{T}}\right)$.
\end{proof}

\subsection{Miscellaneous Technical Lemmas}
\label{app:tech}
\begin{lemma}\label{lem:trig}
   For a Riemannian manifold $\M$ with sectional curvature in $[\k,K]$ and a geodesic triangle on $\M$ with vertices $\x,\y,\z$. If $K>0$, we require the maximum side length is smaller than $\frac{\pi}{\sqrt{K}}$. If $\kappa<0$, we assume
   \[
   \max\{d(\x,\z),d(\y,\z)\}\leq\frac{1}{\sqrt{-\kappa}},
   \]
   then we have
   \[
   d(\x,\y)\leq 2\cdot\|\invexp_{\z}\x-\invexp_{\z}\y\|.
   \]
\end{lemma}
\begin{proof}
    By Proposition B.$2$ in \citet{ahn2020nesterov} and Rauch Comparison Theorem, we have
    \[
    d(\x,\y)\leq\left\{
    \begin{array}{ll}
    \frac{\sinh(\sqrt{-\kappa}\max\{d(\x,\z),d(\y,\z)\})}{\sqrt{-\kappa}\max\{d(\x,\z),d(\y,\z)\}}\cdot \|\invexp_{\z}\x-\invexp_{\z}\y\| & \kappa<0\\
    \|\invexp_{\z}\x-\invexp_{\z}\y\| & \kappa\geq 0.\\
    \end{array}
\right.
    \]
    Since $\frac{\sinh x}{x}$  is monotonically increasing with respect to $x$ and 
    \[
    \max\{d(\x,\z),d(\y,\z)\}\leq\frac{1}{\sqrt{-\kappa}},
    \]
    we have 
    \[
    \frac{\sinh(\sqrt{-\kappa}\max\{d(\x,\z),d(\y,\z)\})}{\sqrt{-\kappa}\max\{d(\x,\z),d(\y,\z)\}}\leq \frac{\sinh 1}{1}\leq 2,
    \]
    which completes the proof.
\end{proof}
\begin{lemma}\citep[Theorem 7.11]{lee2018introduction}, \citep[Lemma 11]{wang2023riemannian}\label{lem:gauss}
For a Riemannian manifold $\M$ with sectional curvature in $[\k,K]$, denote $\Lambda(s,t):[0,1]\times[0,1]\rightarrow\M$ to be a rectangle map and $\l$ to be the geodesic loop connecting $\Lambda(0,0)$, $\Lambda(0,1)$, $\Lambda(1,1)$ and $\Lambda(1,0)$. We also define $K_m=\max(|\k|,|K|)$, vector fields $S(s,t)=\Lambda_{\star}\frac{\partial}{\partial s}(s,t)$ and $T(s,t)=\Lambda_{\star}\frac{\partial}{\partial t}(s,t)$. 
Then for any $\u\in T_{\Lambda(0,0)}\M$, we have
\[
\|\G_{\l}\u-\u\|\leq 12 K_m\|\u\|\cdot \int_0^1\int_0^1\|T\|\cdot\|S\|ds dt,
\]
where $\G_{\l}$ is the parallel transport along the geodesic loop $\l$.

\end{lemma}
\begin{defn}\label{def:jacobi}
       We define
    \[
    \bm{S}(K, t)=\left\{\begin{array}{ll}
\frac{\sin (\sqrt{K} t)}{\sqrt{K}} & K>0 \\
t & K \leq 0,
\end{array}\right.
    \]
and
    \[
    \bm{s}(\kappa, t)=\left\{\begin{array}{ll}
\frac{\sinh (\sqrt{-\kappa} t)}{\sqrt{-\kappa}} & \kappa<0 \\
t & \kappa \geq 0.
\end{array}\right.
    \]
\end{defn}
\begin{lemma}\citep{lee2018introduction}\label{lem:jacobi}
    Let the sectional curvature of a Riemannian manifold $\M$ be in $[\kappa,K]$, $\gm(t):[0,s]\rightarrow\M$ be a geodesic with unit velocity
    , and $J$ be a Jacobi field along $\gm(t)$. When $K>0$, we assume the length of $\gm(t)$ is smaller than $\frac{\pi}{\sqrt{K}}$.
    Then
    \[
    \bm{S}(K, t)\|\nabla_{\dot{\gm}}J(\gm(0))\|\leq\|J(\gm(t))\|\leq \bm{s}(\k, t)\|\nabla_{\dot{\gm}}J(\gm(0))\|,
    \]
    where $\bm{s}(\k,t)$ and $\bm{S}(K,t)$ are defined in  Definition \ref{def:jacobi}.
\end{lemma}

\begin{lemma}\citep{wang2023riemannian}\label{lem:jacobi1}
    We denote $K_m=\max(|\k|,|K|)$. For any $0\leq t\leq\frac{1}{\sqrt{K_m}}$, we have $\frac{\bm{s}(\k,t)}{\bm{S}(K,t)}\leq 3$, where $\bm{s}(\k,t)$ and $\bm{S}(K,t)$ are defined in  Definition \ref{def:jacobi}.
\end{lemma}





\begin{lemma}
    \label{lem:cos2}\citep{alimisis2020continuous}
    Let $\M$ be a Riemannian manifold with sectional curvature upper bounded by $K$. Consider a geodesic triangle with side lengths $a,b,c$ such that
    \[
    b+\min\{a,c\}<\left\{
    \begin{array}{ll}
        \frac{\pi}{\sqrt{K}} & K>0 \\
        \infty & K\leq 0.
    \end{array}
\right.
    \]
    Then we have
    \[
a^{2} \geq {\sigma}(K, b+\min\{a,c\}) b^{2}+c^{2}-2 b c \cos A
\]
where 
\[
\sigma(K,\tau)\coloneqq\left\{
    \begin{array}{ll}
        \sqrt{K} \tau \operatorname{cot}(\sqrt{K} \tau) & K> 0 \\
        1 & K\leq 0.
    \end{array}
\right.
\]
\end{lemma}
\begin{remark}
Lemma \ref{lem:cos2} is indeed a variant of Corollary 2.1 of \citet{alimisis2020continuous}. The original version therein states: given a geodesic triangle with vertices $\x$, $\y$, and $\z$, and corresponding edge lengths $a$, $b$, and $c$. There exists a point $\q \in \overline{\x\z}$ such that
    \[
    \begin{split}
        a^2\geq\sigma(K,d(\y,\q))b^2+c^2-2bc\cos A.
    \end{split}
    \]
    By the triangle inequality
    \[
    \begin{split}
        d(\y,\q)\leq\min\{d(\x,\y)+d(\x,\z),d(\x,\z)+d(\y,\z)\}=b+\min\{a,c\}.
    \end{split}
    \]
    Thus, by the monotonicity of $\sigma(K,\cdot)$, we have
    \[
    \begin{split}
        a^2\geq&\sigma(K,d(\y,\q))b^2+c^2-2bc\cos A\\
        \geq &\sigma(K,b+\min\{a,c\})b^2+c^2-2bc\cos A.
    \end{split}
    \]
\end{remark}

\begin{lemma}\citep{alimisis2021momentum}\label{lem:hessian}
Assume $\M$ is a Riemannian manifold with sectional curvature in $[\k,K]$. For a geodesic triangle on $\M$ with vertices $\x,\y,\z$ such that 

\[
    \tau\coloneqq d(\x,\z)+\min\{d(\x,\y),d(\y,\z)\}<\left\{
    \begin{array}{ll}
        \frac{\pi}{\sqrt{K}} & K>0 \\
        \infty & K\leq 0.
    \end{array}
\right.
\]
Then
 \begin{enumerate}[label=\arabic*),leftmargin=*]
\item \[
\|\invexp_{\x}\y-\G_{\z}^{\x}\invexp_{\z}\y\|\leq\zeta(\kappa,\tau)\cdot d(\x,\z)
\]
\item \[
\|\invexp_{\x}\y-\G_{\z}^{\x}\invexp_{\z}\y-\invexp_{\x}\z\|\leq\max\{\zeta(\kappa,\tau)-1,1-\sigma(K,\tau)\}\cdot d(\x,\z)
\]
 \end{enumerate}
where
\[
\zeta(\kappa,\tau)\coloneqq\left\{
    \begin{array}{ll}
        \sqrt{-\kappa} \tau \operatorname{coth}(\sqrt{-\kappa} \tau) & \k< 0 \\
        1 & \k\geq 0,
    \end{array}
\right.
\]
and $\sigma(K,\tau)$ is defined in Lemma \ref{lem:cos2}.
\end{lemma}
\begin{lemma}
\label{lem:cos1}
\citep[Lemma 5]{zhang2016first}. Let $\mathcal{M}$ be a Riemannian manifold with sectional curvature lower bounded by $\kappa$. Consider a geodesic triangle fully lies within $\M$ with side lengths $a, b, c$, we have
\[
a^{2} \leq {\zeta}(\kappa, c) b^{2}+c^{2}-2 b c \cos A
\]
where ${\zeta}(\kappa, c)$ is defined in Lemma \ref{lem:hessian}.

\end{lemma}
\subsection{Bounding the Holonomy Effect on a Geodesic Triangle}
\label{app:holo}
\begin{lemma}\label{lem:holo}
For a Riemannian manifold $\M$ with sectional curvature in $[\k,K]$ and a geodesic triangle $\triangle \x\y\z$ on $\M$, we denote $K_m=\max\{|\k|,|K|\}$. Then as long as
\[
\min\{d(\x,\y)+d(\y,\z),d(\x,\z)+d(\y,\z)\}\leq\frac{1}{\sqrt{K_m}},
\]
for any $\u\in T_{\x}\M$, we have
\[
\|\G_{\z}^{\x}\G_{\y}^{\z}\G_{\x}^{\y}\u-\u\|\leq 36K_m\|\u\|\cdot\min\{d(\x,\y)+d(\y,\z),d(\x,\z)+d(\y,\z)\}\cdot d(\y,\z).
\]
\end{lemma}
\begin{proof}
    We define $\Lambda(s,t):[0,1]\times[0,1]\rightarrow\M$ to be a rectangle map such that
\[
\Lambda(s,t)\coloneqq\expmap_{\gm_1(s)}(t\invexp_{\gm_1(s)}\gm_2(s))
\]
where $\gm_1(s)$ and $\gm_2(s)$ are geodesics which satisfy $\gm_1(0)=\Lambda(0,0)\coloneqq \x$, $\gm_1(1)=\Lambda(1,0)\coloneqq\w$, $\gm_2(0)=\La(0,1)\coloneqq\y$ and $\gm_2(1)=\La(1,1)\coloneqq\z$. Also, we denote
\[
\begin{split}
    &S(s,t)=\Lambda_{\star}\frac{\partial}{\partial s}(s,t)\\
    &T(s,t)=\Lambda_{\star}\frac{\partial}{\partial t}(s,t).
\end{split}
\]

We use $\l$ to denote the geodesic loop starts at $\x$ and consists of geodesic segments $\overline{\x\y}$, $\overline{\y\z}$, $\overline{\z\w}$ and $\overline{\w\x}$, then by Lemma \ref{lem:gauss}, we have
\begin{equation}\label{eq:holo1}
    \|\G_{\l}\u-\u\|\leq 12 K_m\|\u\|\cdot \int_0^1\int_0^1\|T\|\cdot\|S\|ds dt
\end{equation}
By the triangle inequality, $T(s,t)$ simultaneously satisfies
\[
\begin{split}  
\|T(s,t)\|&=d(\gm_1(s),\gm_2(s)) \\
\leq& d(\gm_1(s),\x)+d(\x,\y)+d(\y,\gm_2(s))\\
\leq& d(\w,\x)+d(\x,\y)+d(\y,\z)
\end{split}
\]
and
\[
\begin{split}  
\|T(s,t)\|&=d(\gm_1(s),\gm_2(s)) \\
\leq& d(\gm_1(s),\w)+d(\w,\z)+d(\z,\gm_2(s))\\
\leq& d(\x,\w)+d(\w,\z)+d(\z,\y),
\end{split}
\]
we have
\[
\|T(s,t)\|\leq \min\{d(\w,\x)+d(\x,\y)+d(\y,\z),d(\x,\w)+d(\w,\z)+d(\z,\y)\}
\]
Note that $S$ is a Jacobi field with $\|S(s,0)\|=d(\x,\w)$ and $\|S(s,1)\|=d(\y,\z)$. Now as we set $\w=\x$,
\begin{equation}\label{eq:holo2}
    \begin{split}
    &\|T(s,t)\|\leq\min\{d(\x,\y)+d(\y,\z),d(\x,\z)+d(\y,\z)\}\\
    &\|S(s,0)\|=0\\
    &\|S(s,1)\|=d(\y,\z).
\end{split}
\end{equation}
We define a unit speed geodesic $\gm(t)$ such that $\gm(0)=\gm_1(s)$ and $\gm(b)=\gm_2(s)$, where $b\coloneqq d(\gm_1(s),\gm_2(s))$. Then we find $J(\gm(t))\coloneqq S(s,t/b)$ is a Jacobi field associated with the geodesic $\gm(t)$. By Lemma \ref{lem:jacobi}, for any $t\in[0,b]$, we have
\[
\bm{S}(K,b)\|\nabla_{\dot{\gm}}J(\gm(0))\|\leq \|J(\gm(b))\|
\]
and
\[
\|J(\gm(t))\|\leq \bm{s}(K,t)\|\nabla_{\dot{\gm}}J(\gm(0))\|\leq \bm{s}(K,b)\|\nabla_{\dot{\gm}}J(\gm(0))\|.
\]
Combining the above two inequalities yields
\[
\|J(\gm(t))\| \leq\frac{\bm{s}(\kappa,b)}{\bm{S}(K,b)}\|J(\gm(b))\|=\frac{\bm{s}(\kappa,b)}{\bm{S}(K,b)}\|S(s,1)\|=\frac{\bm{s}(\kappa,b)}{\bm{S}(K,b)} d(\y,\z)
\]
holds for any $t\in[0,b]$,
which is equivalent to
\begin{equation}\label{eq:holo3}
    \lt\|S\lt(s,\frac{t}{b}\rt)\rt\|\leq\frac{\bm{s}(\kappa,b)}{\bm{S}(K,b)}d(\y,\z).
\end{equation}
holds for any $t\in[0,b]$. By combining Equations \eqref{eq:holo1}
, \eqref{eq:holo2} and \eqref{eq:holo3}, we find

\[
\begin{split}
    \|\G_{\z}^{\x}\G_{\y}^{\z}\G_{\x}^{\y}\u-\u\|\leq &12K_m\|\u\|\cdot\min\{d(\x,\y)+d(\y,\z),d(\x,\z)+d(\y,\z)\}\cdot \frac{\bm{s}(\kappa,\|T(s,t)\|)}{\bm{S}(K,\|T(s,t)\|)}d(\y,\z)\\
    \leq& 36K_m\|\u\|\cdot\min\{d(\x,\y)+d(\y,\z),d(\x,\z)+d(\y,\z)\}\cdot d(\y,\z),
\end{split}
\]
where in the second inequality, we apply 
\[
\|T(s,t)\|\leq\min\{d(\x,\y)+d(\y,\z),d(\x,\z)+d(\y,\z)\}\leq\frac{1}{\sqrt{K_m}}
\]
and Lemma \ref{lem:jacobi1}.
\end{proof}
\begin{remark}
In the literature \citep{karcher1977riemannian,sun2019escaping}, a proposition similar to our Lemma \ref{lem:holo} is presented as:
\begin{equation}
    \|\G_{\z}^{\x}\G_{\y}^{\z}\G_{\x}^{\y}\u-\u\| \leq \tC \cdot d(\x,\y) \cdot d(\y,\z) \|\u\|,
\end{equation}
which holds for some constant $\tC$ and for all $\x, \y, \z \in \M$. By comparison, Lemma \ref{lem:holo} explicitly elucidates the dependence of $\tC$ on both the diameter of the geodesic triangle $\triangle \x\y\z$ and the Riemannian curvature.
\end{remark}
\end{document}

%% file: cmd.tex
\usepackage{paralist,amsmath, amssymb}
\usepackage{algorithm,algorithmic}
\usepackage{mathtools}
\usepackage{amsthm}
\usepackage{calc}

 \usepackage{graphicx}
 \usepackage{float}
\usepackage{xcolor}
\usepackage{graphics}
\usepackage{xspace}
\usepackage{multicol}

 \usepackage{amssymb, multirow, paralist, color}
 \usepackage{textcase, booktabs,makecell}
\usepackage{enumitem}
\def \La {\Lambda}

\def \l {\boldsymbol\ell}

\def \tC {\tilde{C}}

\def \k {\kappa}

\def \y {\mathbf{y}}

\def \x {\mathbf{x}}

\def \gm {\gamma}

\def \D {\mathcal{D}}
\def \z {\mathbf{z}}
\def \u {\mathbf{u}}

\def \w {\mathbf{w}}
\def \R {\mathbb{R}}

\def \l {\boldsymbol\ell}

\def \expmap {{\textnormal{Exp}}}
\def \invexp {\textnormal{Exp}^{-1}}

\def \G {\Gamma}

\def \q {\mathbf{q}}
\def \v {\mathbf{v}}
\def \M {\mathcal{M}}
\def \c {\mathbf{c}}

\def \q {\mathbf{q}}

\def \bs {\bar{\sigma}}
\def \bzt {\bar{\zeta}}
\def \bz {\bar{\z}}

\def \a {\mathbf{a}}
\def \b {\mathbf{b}}

\def \B {\mathcal{B}}

\def \M {\mathcal{M}}

\def \rg {\mathsf{grad}\kern 0.1em}
\def \rh {\mathsf{Hess}\kern 0.1em}

\def \X {\mathcal{X}}
\def \Y {\mathcal{Y}}
\def \Z {\mathcal{Z}}

\def \la {\left\langle}
\def \ra {\right\rangle}
\def \bx {\bar{\x}}
\def \by {\bar{\y}}

\def \lt {\left}
\def \rt {\right}
\def \sumt {\sum_{t=1}^T}

\def \tx {\tilde{\x}}
\def \ty {\tilde{\y}}
\def \tz {\tilde{\z}}

\newtheorem{ass}{Assumption}
\newtheorem{thm}{Theorem}
\newtheorem{cor}{Corollary}
\newtheorem{lemma}{Lemma}
\newtheorem{defn}{Definition}
\newtheorem{remark}{Remark}

\usepackage{microtype}
\usepackage{graphicx}
\usepackage{booktabs} 